\documentclass[a4paper]{amsart}
\usepackage[margin=1in]{geometry}
\usepackage{multicol}
\usepackage{fancyhdr}
\usepackage{stmaryrd} 
\usepackage{nicefrac, xfrac} 
\usepackage{amssymb} 
\usepackage{bbold} 
\usepackage{float}
\usepackage{amsthm}
\usepackage{mathtools} 
\usepackage{thmtools} 
\usepackage{comment}

\usepackage{tikz}
\usetikzlibrary{calc}
\usetikzlibrary{arrows}
\usetikzlibrary{decorations.markings}
\usetikzlibrary{decorations.pathreplacing}
\usetikzlibrary{arrows,shapes,positioning}
\usetikzlibrary{shapes.geometric}
\usetikzlibrary{patterns}
\tikzstyle directed=[postaction={decorate,decoration={markings,
    mark=at position #1 with {\arrow{>}}}}]
\tikzstyle rdirected=[postaction={decorate,decoration={markings,
    mark=at position #1 with {\arrow{<}}}}]
\tikzset{anchorbase/.style={baseline={([yshift=-0.5ex]current bounding box.center)}}}

\tikzset{
    partial ellipse/.style args={#1:#2:#3}{
        insert path={+ (#1:#3) arc (#1:#2:#3)}
    }
  }

\usepackage{tikz-cd}

\usepackage{hyperref}
\usepackage[capitalize]{cleveref}
\crefname{figure}{Figure}{Figures}

\newcommand{\Q}{\mathbb{Q}}
\newcommand{\R}{\mathbb{R}}
\newcommand{\C}{\mathbb{C}}
\newcommand{\Z}{\mathbb{Z}}
\newcommand{\N}{\mathbb{N}}

\DeclareMathOperator{\End}{End}
\DeclareMathOperator{\END}{END}

\DeclareMathOperator{\PSL}{PSL}
\DeclareMathOperator{\PGL}{PGL}

\let\tilde=\widetilde

\newcommand{\qn}[1]{\lbrace #1 \rbrace}
\newcommand{\tqn}[1]{\tilde{\lbrace #1 \rbrace}}
\newcommand{\qbin}[2]{\begin{Bmatrix} #2 \\ #1\end{Bmatrix}}
\newcommand{\xWeb}{x{\bf \mathrm{Web}}}
\newcommand{\nWeb}{n{\bf \mathrm{Web}}}
\newcommand{\iWeb}{\infty{\bf \mathrm{Web}}}

\newcommand{\Schur}{\mathcal{S}}

\theoremstyle{plain}
\newtheorem{theorem}{Theorem}[section]
\newtheorem{lemma}[theorem]{Lemma}
\newtheorem{corollary}[theorem]{Corollary}
\newtheorem{remark}[theorem]{Remark}
\newtheorem{notation}[theorem]{Notation}
\newtheorem{definition}[theorem]{Definition}
\newtheorem{example}[theorem]{Example}

\newtheorem{proposition}[theorem]{Proposition}
\newtheorem{question}[theorem]{Question}

\title{$q$-rationals, link invariants and webs}
\author{Perrine Jouteur}
\address{Universit\'e de Reims Champagne Ardenne,
Laboratoire de Math\'e\-ma\-tiques, CNRS UMR9008,
Moulin de la Housse - BP 1039,
51687 Reims cedex 2,
France}
\email{perrine.jouteur@univ-reims.fr}

\author{Hoel Queffelec}
\address{France-Australia Mathematical Sciences and Interactions ANU - CNRS International Research Laboratory \\ Australia}
\email{hoel.queffelec@cnrs.fr}

\date{}

\begin{document}

\maketitle

\begin{abstract} We investigate the role of $q$-rationals in the context of link invariants. This leads us to introduce and study web categories that $q$-deform Deligne's categories. 
\end{abstract}

\section{Introduction}

The most famous polynomial knot and link invariants, namely the Alexander polynomial~\cite{Alexander} and the Jones polynomial~\cite{Jones}, can be united by the (2-variable) HOMFLY-PT polynomial~\cite{HOMFLY,PT}. These invariants can be computed from a diagram of the knot, via so-called skein relations. A key relation sets the value for the unknot. For the Reshetikhin-Turaev $\mathfrak{sl}_n$-invariant~\cite{RT}, of which the Jones  polynomial is the $n=2$ incarnation, one gets the value $[n]=\frac{q^n-q^{-n}}{q-q^{-1}}$ (to be read \emph{quantum n}), which is the graded dimension of the vector representation of $U_q(\mathfrak{sl}_n)$. This generalizes for the HOMFLY-PT polynomial in a 2-variable rational fraction $\frac{a-a^{-1}}{q-q^{-1}}$, and we safely ignore the Alexander polynomial from this discussion as it requires further renormalizations and modified traces \cite{Viro,SarAlexander}.

It is a very tempting project to try to generalize these invariants to non-integral values of $n$. Part of it has been done by Vogel~\cite{Vogel}, by considering the value of the unknot as a parameter. But then it is unclear what admissible values can be used for this parameter that would bear a special meaning.

As implicitly conveyed in the previous paragraphs, the categories of representations of quantum groups $U_q(\mathfrak{gl}_n)$ (and super-groups $U_q(\mathfrak{gl}_{r|s})$) play a central role in several definitions for the knot invariants, via the diagrammatics of web categories. It is most natural to try replacing the discrete parameter $n$ by a continous one. The most naive version would simply consist in copying the definition for the quantum dimension, replacing $[n]=\frac{q^n-q^{-n}}{q-q^{-1}}$ by a yet-to-interpret $[x]=\frac{q^x-q^{-x}}{q-q^{-1}}$. This simple definition of $[x]$, which is in line with Brundan's work~\cite{Brundan}, introduces non integer powers of $q$, which are difficult to work with (at least at a combinatorial level) and do not specialize to holomorphic complex functions if one wishes to consider $q$ as a complex variable. 

In that regard, the recent definition of the notion of $q$-rationals by Morier-Genoud and Ovsienko \cite{MGO-2020} carries new hope to define meaningful knot and link invariants interpolating the Reshetikhin-Turaev polynomials. Indeed, they assign to each $x\in \Q$ a rational fraction that recovers (up to some rescalings) the quantum integers if $x\in \Z$, and generalizes it in a relevant way from the algebraic and combinatorial perspective. Furthermore, one can use a limit argument to also $q$-deform any $x\in \R$ \cite{MGOr}. These $q$-real numbers were already used in the context of link invariants, in particular to compute the Kauffman bracket polynomial of rational tangles \cite{sikora_tangle_2024,Kogiso_Wakui}.

In this paper, we show that the categories of webs used to define the HOMFLY-PT polynomial admit natural specializations to categories that are linear over a quadratic extension of $\Q(q)$, and that assign to the stabilized unknot the $q$-rational $\qn{x}$. We give a presentation for these categories in Proposition~\ref{prop:xWeb_presentation}, and extract a link invariant from them, that is naturally a specialization of the HOMFLY-PT polynomial.

These web categories interpolate between the $\mathfrak{gl}_n$-web categories, and provide a quantum analog to Deligne's categories $\mathrm{Rep}(\mathfrak{gl}_t)$. We believe that these categories should play an interesting role in representation theory. In Section~\ref{sec:Deligne}, we argue that the definition of the categories $\xWeb$ is essentially dictated by a few natural conditions. Then we briefly investigate in Section~\ref{sec:derivation} some interesting properties of these families of knot invariants related to continuity and derivability.

We end up with an investigation of an alternative construction that uses so-called \emph{left} $q$-rationals~\cite{BBL}.

A perspective that we find quite exciting and we would like to finally mention is the question of categorifying the invariant: might it be that a categorified $x$-invariant brings something genuinely new?

\vspace{.5cm}

\noindent {\bf Code:} In the early stages of this project, before we realized that the invariant is a specialization of the HOMFLY-PT polynomial, we developed code to compute it. The fact that it systematically failed at distinguishing knots with the same HOMFLY-PT polynomial led us to the results in this paper. The code is available on a public repository\footnote{See \url{https://plmlab.math.cnrs.fr/hqueffel/ratiknot_public/}}. The backend functions providing computations with webs might be interesting in other contexts.

\vspace{.5cm}

\noindent {\bf Acknowledgements:} We wish to thank Asilata Bapat, Ben Elias, Mikhail Khovanov, Joan Licata, Tony Licata, Adam Sikora and Daniel Tubbenhauer for interesting discussions. The work of H.Q. was partially funded by the Agence Nationale de la Recherche, under the projects \emph{CaGeT} (ANR-25-CE40-4162) and \emph{NASQI3D} (ANR-24-CE40-7252).

\section{$q$-real numbers} \label{sec:qrat}

We review notations and definitions of $q$-deformed rational numbers, for an indeterminate $q$. Recall first the $q$-deformed integers and binomial coefficients (normalized to have actual polynomials in $q$). Notice that in order to stay closer to the literature on knot invariants and web categories, the variable $q$ in~\cite{MGO-2020} and subsequent references is replaced by $q^2$. We thus define:

\[
\qn{n} := \frac{q^{2n}-1}{q^2-1}, \text{ and } \qbin{k}{n} := \frac{\qn{n}\qn{n-1}\cdots \qn{n-k+1}}{\qn{k}\qn{k-1}\cdots \qn{1}}.
\]

\begin{definition}[\cite{MGO-2020}]
Let $x \in \Q$, and let $[a_1,a_2,\cdots,a_n]$ be its even continued fraction representation. The $q$-analogue of $x$ is 
\begin{equation}
    \qn{x} = \qn{a_1} + \dfrac{q^{2a_1}}{\{a_2\}_{q^{-2}} + \dfrac{q^{-2a_2}}{ \qn{a_3} + \dfrac{q^{2a_3}}{\ddots + \dfrac{q^{2a_{n-1}}}{\{a_n\}_{q^{-2}}}}}}.
\end{equation}
\end{definition}

This $q$-deformation is designed to be equivariant for a $q$-deformed action of $\PGL_2(\Z)$ on rational numbers \cite{Leclere_modular,Jouteur_sym}. In particular, $q$-rational numbers satisfy
\begin{equation}
    \label{shift_equivariance}
    \qn{x+1} = q^2\qn{x} + 1 \text{ and } \left\lbrace\frac{1}{x}\right\rbrace = \frac{1}{\qn{x}_{q^{-1}}}.
\end{equation}

\begin{notation} \label{not:delta_x}
For a rational number $x$, let $\delta_x := \qn{x} - \qn{x-1}$. Note that for $x=n\in \Z$, this reduces to $\delta_n = q^{2(n-1)}$.
\end{notation}

We also have a generalized version of $q$-binomial coefficients.

\begin{definition}
Let $x\in \Q$ and let $k \in \N$. Define
\begin{equation}
    \label{qbinomials}
    \qbin{k}{x} := \frac{\qn{x}\qn{x-1}\cdots \qn{x-k+1}}{\qn{k}\qn{k-1}\cdots \qn{1}}.
\end{equation}
\end{definition}

\begin{definition}[\cite{MGOr}]
Let $x\in \R\setminus \Q$, and let $(x_n)_n$ be a sequence of rational numbers whose limit is $x$. The $q$-deformation of $x$ is defined to be the $q$-adic limit of the sequence $(\qn{x_n})_n$.\\
\noindent In other words, for each $n\in \N$, let $\sum_k c_k^{(n)}q^k$ be the Taylor expansion at $q=0$ of $\qn{x_n}$. Then 
\begin{equation}
    \label{qirrationals}
    \qn{x} := \sum_{k}c_k q^k \text{ where } c_k = \lim_{n\mapsto +\infty} c_k^{(n)}.
\end{equation}

\end{definition}

\begin{remark}
For all $x\in \R$, we have $\qn{x} \underset{q\mapsto 1}{\rightarrow} x$. See \cite{Etingof2025,MGOr,BBL} for more details.
\end{remark}

\section{Links and web categories} \label{sec:webs}

\subsection{Knot and links}

A link is a smooth embedding of a finite collection of disjoint circles in $\R^3$, up to isotopy. It is said to be framed if it comes with a (homotopy class of) non-singular normal vector field. One can represent framed links by pictures called link diagrams, which are generic projections of links on the plane, keeping track of the over or under crossings. See \cite{KRT} for more details about links. Two link diagrams will represent the same framed link if and only if they differ by a finite number of local moves called the Reidemeister moves. 

\begin{definition}[framed Reidemeister moves]
\label{def:reidemeister}
\[
\begin{tikzpicture}[anchorbase]
\begin{scope}
\draw[semithick] (0,0) -- (0,2);
\node at(0.5,1) {$\overset{R'_1}{\longleftrightarrow}$};
\draw[semithick] (1,0) -- (1,0.3);
\draw[semithick] (1.08,0.5) to[out=70,in=180] (1.35,0.75) to[out=0,in=90] (1.6,0.5) to[out=-90,in=0] (1.3,0.2) to[out=180,in=-90] (1,0.5);
\draw[semithick] (1,0.5) -- (1,1.3);
\draw[semithick] (1,1.3) to[out=90,in=180] (1.3,1.8) to[out=0,in=90] (1.6,1.5) to[out=-90,in=0] (1.3,1.2) to[out=180,in=-80] (1.1,1.4);
\draw [semithick] (1,1.7) -- (1,2);
\end{scope}
\begin{scope}[shift={(3,0)}]
\draw[semithick] (-0.2,0) -- (-0.2,2);
\draw[semithick] (0,0) -- (0,2);
\node at(0.5,1) {$\overset{R_2}{\longleftrightarrow}$};
\draw[semithick] (1,0) -- (1,0.3) to[out=90,in=-90] (1.3,1) to[out=90,in=-90] (1,1.7) -- (1,2) ;
\draw[semithick] (1.3,0) -- (1.3,0.3) to[out=90,in=-60] (1.2,0.6);
\draw[semithick] (1.1,0.7) to[out=120,in=-90] (1,1) to[out=90,in=-120] (1.1,1.3);
\draw[semithick] (1.2,1.4) to[out=60,in=-90] (1.3,1.7) -- (1.3,2); 
\end{scope}
\begin{scope}[shift={(6,0)}]
\draw[semithick] (0,0) -- (0.6,2);
\draw[semithick] (0,2) -- (0.25,1.1);
\draw[semithick] (0.3,2) -- (0.4,1.55);
\draw[semithick] (0.3,0) -- (0.6,0.7);
\draw[semithick] (0.6,0.7) to[out=80,in=-60] (0.5,1.4);
\draw[semithick] (0.32,0.88) -- (0.45,0.45);
\draw[semithick] (0.5,0.35)-- (0.6,0);
\node at(1.3,1) {$\overset{R_3}{\longleftrightarrow}$};
\draw[semithick] (2,0) -- (2.6,2);
\draw[semithick] (2,2) -- (2.12,1.61);
\draw[semithick] (2.2,1.45) -- (2.28,1.15);
\draw[semithick] (2.35,1) -- (2.6,0);
\draw[semithick] (2.3,2) -- (2,1.1);
\draw[semithick] (2,1.1) to[out=-100,in=110] (2.12,0.6);
\draw[semithick] (2.2,0.5) -- (2.3,0);
\end{scope}
\end{tikzpicture}
\]
\end{definition}

The HOMFLY-PT polynomial is a link invariant that will play a central role in this paper. One possible definition uses the language of webs (see for example \cite{QS_note,QS_Schur,Brundan} and references therein), which we will introduce now. Webs originate from the work of Kuperberg~\cite{Kup}, and their modern study stems from work of Cautis-Kamnitzer-Morrison~\cite{CKM}. Web categories are diagrammatic categories that originate from the study of intertwiners for representations of quantum groups, which have been central in quantum topology since the work of Reshetikhin and Turaev~\cite{RT}.

Classical web categories thus depend on the choice of a quantum group. In the most classical case, one chooses $\mathfrak{gl}_n$ as underlying Lie algebra, and the category depends on the parameter $n$. One of the main goals of our paper is to give a motivated definition for web categories for non-integral parameters $x\in \Q$. (Then one could take limits and extend this to $x\in \R$.) We will first review classical definitions and highlight their associated link invariants, before we introduce new ones in Section~\ref{sec:interpolating}.

\subsection{Upward webs}

We start by introducing upwards webs, which will be the building blocks of all later web categories.

\begin{definition} A purely upward web is an oriented trivalent graph with edges labeled in $\N^*$, embedded in $\R^2$ with an upward orientation at any generic height. Moreover, it must satisfy a flow condition at each vertex :
\[
\begin{tikzpicture}[anchorbase,scale=.8]
\draw [semithick] (0,0) -- (0,1);
\node at (0,-.3) {$k+l$};
\draw [semithick,->] (0,1) -- (-1,2);
\node at (-1,2.3) {$k$};
\draw [semithick,->] (0,1) -- (1,2);
\node at (1,2.3) {$l$};
\end{tikzpicture}
\qquad
\begin{tikzpicture}[anchorbase,scale=.8]
\draw [semithick] (-1,0) -- (0,1);
\node at (-1,-.3) {$k$};
\draw [semithick] (1,0) -- (0,1);
\node at (1,-.3) {$l$};
\draw [semithick,->] (0,1) -- (0,2);
\node at (0,2.3) {$k+l$};
\end{tikzpicture}
\]

Webs are subject to isotopy (respecting the height) and the following list of relations\footnote{In the relations, it might happen that edges labeled $0$ appear. The convention is that these can be erased. On the other hand, if a web with a negative label appears, it is set to zero.}:
\begin{equation}\label{eq:triwebs}
\begin{tikzpicture}[anchorbase,  decoration={markings,mark=at position 0.5 with {\arrow{>}}}]
    \draw [semithick, postaction={decorate}] (0,0) to [out=90,in=-150] (.5,1);
    \draw [semithick,postaction={decorate}] (1,0) to [out=90,in=-30] (.5,1);
    \draw [semithick,postaction={decorate}] (.5,1) to [out=90,in=-150] (1,1.5);
    \draw [semithick,postaction={decorate}] (2,0) to [out=90,in=-30] (1,1.5);
    \draw [semithick,->] (1,1.5) -- (1,2);
    \node at (0,-.3) {\small $k$};
    \node at (1,-.3) {\small $l$};
    \node at (2,-.3) {\small $r$};
\end{tikzpicture}
\quad = \quad
\begin{tikzpicture}[anchorbase,  decoration={markings,mark=at position 0.5 with {\arrow{>}}}]
    \draw [semithick, postaction={decorate}] (2,0) to [out=90,in=-30] (1.5,1);
    \draw [semithick,postaction={decorate}] (1,0) to [out=90,in=-150] (1.5,1);
    \draw [semithick,postaction={decorate}] (1.5,1) to [out=90,in=-30] (1,1.5);
    \draw [semithick,postaction={decorate}] (0,0) to [out=90,in=-150] (1,1.5);
    \draw [semithick,->] (1,1.5) -- (1,2);
    \node at (0,-.3) {\small $k$};
    \node at (1,-.3) {\small $l$};
    \node at (2,-.3) {\small $r$};
\end{tikzpicture}
\quad, \quad
\begin{tikzpicture}[anchorbase,  decoration={markings,mark=at position 0.5 with {\arrow{>}}}]
    \draw [semithick, ->] (.5,1) to [out=150,in=-90] (0,2);
    \draw [semithick,->] (.5,1) to [out=30,in=-90] (1,2);
    \draw [semithick,->] (1,.5) to [out=150,in=-90] (.5,1);
    \draw [semithick,->] (1,.5) to [out=30,in=-90] (2,2);
    \draw [semithick,postaction={decorate}] (1,0) -- (1,.5);
    \node at (0,2.3) {\small $k$};
    \node at (1,2.3) {\small $l$};
    \node at (2,2.3) {\small $r$};
\end{tikzpicture}
\quad = \quad
\begin{tikzpicture}[anchorbase,  decoration={markings,mark=at position 0.5 with {\arrow{>}}}]
    \draw [semithick, ->] (1.5,1) to [out=150,in=-90] (1,2);
    \draw [semithick,->] (1.5,1) to [out=30,in=-90] (2,2);
    \draw [semithick,->] (1,.5) to [out=30,in=-90] (1.5,1);
    \draw [semithick,->] (1,.5) to [out=150,in=-90] (0,2);
    \draw [semithick,postaction={decorate}] (1,0) -- (1,.5);
    \node at (0,2.3) {\small $k$};
    \node at (1,2.3) {\small $l$};
    \node at (2,2.3) {\small $r$};
\end{tikzpicture}
\end{equation}

\begin{equation}\label{eq:updigon}
\begin{tikzpicture}[anchorbase, decoration={markings,mark=at position 0.5 with {\arrow{>}}}]
\draw [semithick,postaction={decorate}] (0,0) -- (0,.5);
\draw [semithick,postaction={decorate}] (0,.5) to [out=150,in=-90] (-.5,1) to [out=90,in=-150] (0,1.5);
\draw [semithick,postaction={decorate}] (0,.5) to [out=30,in=-90] (.5,1) to [out=90,in=-30] (0,1.5);
\draw [semithick,->] (0,1.5) -- (0,2);
\node at (-.7,1) {\small $k$};
\node at (.7,1) {\small $l$};
\node at (0,-.3) {\small $k+l$};
\end{tikzpicture}
\;=\;
q^{-kl}\qbin{k}{k+l}
\begin{tikzpicture}[anchorbase]
\draw [semithick,->] (0,0) -- (0,2);
\node at (0,-.3) {\small $k+l$};
\end{tikzpicture}
\end{equation}

\begin{equation*}
\begin{tikzpicture}[anchorbase, decoration={markings,mark=at position 0.5 with {\arrow{>}}}]
\draw [semithick,->] (0,0) -- (0,2);
\draw [semithick,->] (1,0) -- (1,2);
\draw [semithick,postaction={decorate}] (0,.35) -- (1,.65);
\draw [semithick,postaction={decorate}] (1,1.35) -- (0,1.65);
\node at (0,-.3) {\small $k$};
\node at (1,-.3) {\small $l$};
\node at (.5,.2) {\small $1$};
\node at (.5,1.7) {\small $1$};
\end{tikzpicture}
\;-\;
\begin{tikzpicture}[anchorbase, decoration={markings,mark=at position 0.5 with {\arrow{>}}}]
\draw [semithick,->] (0,0) -- (0,2);
\draw [semithick,->] (1,0) -- (1,2);
\draw [semithick,postaction={decorate}] (1,.35) -- (0,.65);
\draw [semithick,postaction={decorate}] (0,1.35) -- (1,1.65);
\node at (0,-.3) {\small $k$};
\node at (1,-.3) {\small $l$};
\node at (.5,.2) {\small $1$};
\node at (.5,1.7) {\small $1$};
\end{tikzpicture}
\;=
q^{1-k+l}\qn{k-l}
\;
\begin{tikzpicture}[anchorbase, decoration={markings,mark=at position 0.5 with {\arrow{>}}}]
\draw [semithick,->] (0,0) -- (0,2);
\draw [semithick,->] (1,0) -- (1,2);
\node at (0,-.3) {\small $k$};
\node at (1,-.3) {\small $l$};
\end{tikzpicture}
\end{equation*}

Finally, we have:
\[
\begin{tikzpicture}[anchorbase, decoration={markings,mark=at position 0.5 with {\arrow{>}}}]
\draw [semithick,->] (0,0) -- (0,2);
\draw [semithick,->] (1,0) -- (1,2);
\draw [semithick,->] (2,0) -- (2,2);
\draw [semithick,postaction={decorate}] (1,1.35) -- (0,1.65);
\node at (.5,1.3) {$2$};
\draw [semithick,postaction={decorate}] (2,.35) -- (1,.65);
\node at (1.5,.3) {$1$};
\node at (0,-.3) {$k$};
\node at (1,-.3) {$l$};
\node at (2,-.3) {$r$};
\end{tikzpicture}
\;-\;
\begin{tikzpicture}[anchorbase, decoration={markings,mark=at position 0.5 with {\arrow{>}}}]
\draw [semithick,->] (0,0) -- (0,2);
\draw [semithick,->] (1,0) -- (1,2);
\draw [semithick,->] (2,0) -- (2,2);
\draw [semithick,postaction={decorate}] (1,1.45) -- (0,1.75);
\node at (.5,1.8) {$1$};
\draw [semithick,postaction={decorate}] (2,.85) -- (1,1.15);
\node at (1.5,.8) {$1$};
\draw [semithick,postaction={decorate}] (1,.25) -- (0,.55);
\node at (.5,.2) {$1$};
\node at (0,-.3) {$k$};
\node at (1,-.3) {$l$};
\node at (2,-.3) {$r$};
\end{tikzpicture}
\;+\;
\begin{tikzpicture}[anchorbase, decoration={markings,mark=at position 0.5 with {\arrow{>}}}]
\draw [semithick,->] (0,0) -- (0,2);
\draw [semithick,->] (1,0) -- (1,2);
\draw [semithick,->] (2,0) -- (2,2);
\draw [semithick,postaction={decorate}] (1,.35) -- (0,.65);
\node at (.5,.3) {$2$};
\draw [semithick,postaction={decorate}] (2,1.35) -- (1,1.65);
\node at (1.5,1.3) {$1$};
\node at (0,-.3) {$k$};
\node at (1,-.3) {$l$};
\node at (2,-.3) {$r$};
\end{tikzpicture}
=0;
\]
as well as all symmetries:
\[
\begin{tikzpicture}[anchorbase, decoration={markings,mark=at position 0.5 with {\arrow{>}}}]
\draw [semithick,->] (0,0) -- (0,2);
\draw [semithick,->] (1,0) -- (1,2);
\draw [semithick,->] (2,0) -- (2,2);
\draw [semithick,postaction={decorate}] (2,1.35) -- (1,1.65);
\node at (1.5,1.3) {$2$};
\draw [semithick,postaction={decorate}] (1,.35) -- (0,.65);
\node at (.5,.3) {$1$};
\node at (0,-.3) {$k$};
\node at (1,-.3) {$l$};
\node at (2,-.3) {$r$};
\end{tikzpicture}
\;-\;
\begin{tikzpicture}[anchorbase, decoration={markings,mark=at position 0.5 with {\arrow{>}}}]
\draw [semithick,->] (0,0) -- (0,2);
\draw [semithick,->] (1,0) -- (1,2);
\draw [semithick,->] (2,0) -- (2,2);
\draw [semithick,postaction={decorate}] (2,1.45) -- (1,1.75);
\node at (1.5,1.8) {$1$};
\draw [semithick,postaction={decorate}] (1,.85) -- (0,1.15);
\node at (.5,.8) {$1$};
\draw [semithick,postaction={decorate}] (2,.25) -- (1,.55);
\node at (1.5,.2) {$1$};
\node at (0,-.3) {$k$};
\node at (1,-.3) {$l$};
\node at (2,-.3) {$r$};
\end{tikzpicture}
\;+\;
\begin{tikzpicture}[anchorbase, decoration={markings,mark=at position 0.5 with {\arrow{>}}}]
\draw [semithick,->] (0,0) -- (0,2);
\draw [semithick,->] (1,0) -- (1,2);
\draw [semithick,->] (2,0) -- (2,2);
\draw [semithick,postaction={decorate}] (2,.35) -- (1,.65);
\node at (1.5,.3) {$2$};
\draw [semithick,postaction={decorate}] (1,1.35) -- (0,1.65);
\node at (.5,1.3) {$1$};
\node at (0,-.3) {$k$};
\node at (1,-.3) {$l$};
\node at (2,-.3) {$r$};
\end{tikzpicture}
=0;
\]
\[
\begin{tikzpicture}[anchorbase, decoration={markings,mark=at position 0.5 with {\arrow{>}}}]
\draw [semithick,->] (0,0) -- (0,2);
\draw [semithick,->] (1,0) -- (1,2);
\draw [semithick,->] (2,0) -- (2,2);
\draw [semithick,postaction={decorate}] (0,1.35) -- (1,1.65);
\node at (.5,1.3) {$2$};
\draw [semithick,postaction={decorate}] (1,.35) -- (2,.65);
\node at (1.5,.3) {$1$};
\node at (0,-.3) {$k$};
\node at (1,-.3) {$l$};
\node at (2,-.3) {$r$};
\end{tikzpicture}
\;-\;
\begin{tikzpicture}[anchorbase, decoration={markings,mark=at position 0.5 with {\arrow{>}}}]
\draw [semithick,->] (0,0) -- (0,2);
\draw [semithick,->] (1,0) -- (1,2);
\draw [semithick,->] (2,0) -- (2,2);
\draw [semithick,postaction={decorate}] (0,1.45) -- (1,1.75);
\node at (.5,1.8) {$1$};
\draw [semithick,postaction={decorate}] (1,.85) -- (2,1.15);
\node at (1.5,.8) {$1$};
\draw [semithick,postaction={decorate}] (0,.25) -- (1,.55);
\node at (.5,.2) {$1$};
\node at (0,-.3) {$k$};
\node at (1,-.3) {$l$};
\node at (2,-.3) {$r$};
\end{tikzpicture}
\;+\;
\begin{tikzpicture}[anchorbase, decoration={markings,mark=at position 0.5 with {\arrow{>}}}]
\draw [semithick,->] (0,0) -- (0,2);
\draw [semithick,->] (1,0) -- (1,2);
\draw [semithick,->] (2,0) -- (2,2);
\draw [semithick,postaction={decorate}] (0,.35) -- (1,.65);
\node at (.5,.3) {$2$};
\draw [semithick,postaction={decorate}] (1,1.35) -- (2,1.65);
\node at (1.5,1.3) {$1$};
\node at (0,-.3) {$k$};
\node at (1,-.3) {$l$};
\node at (2,-.3) {$r$};
\end{tikzpicture}
=0;
\]
\[
\begin{tikzpicture}[anchorbase, decoration={markings,mark=at position 0.5 with {\arrow{>}}}]
\draw [semithick,->] (0,0) -- (0,2);
\draw [semithick,->] (1,0) -- (1,2);
\draw [semithick,->] (2,0) -- (2,2);
\draw [semithick,postaction={decorate}] (1,1.35) -- (2,1.65);
\node at (1.5,1.3) {$2$};
\draw [semithick,postaction={decorate}] (0,.35) -- (1,.65);
\node at (.5,.3) {$1$};
\node at (0,-.3) {$k$};
\node at (1,-.3) {$l$};
\node at (2,-.3) {$r$};
\end{tikzpicture}
\;-\;
\begin{tikzpicture}[anchorbase, decoration={markings,mark=at position 0.5 with {\arrow{>}}}]
\draw [semithick,->] (0,0) -- (0,2);
\draw [semithick,->] (1,0) -- (1,2);
\draw [semithick,->] (2,0) -- (2,2);
\draw [semithick,postaction={decorate}] (1,1.45) -- (2,1.75);
\node at (1.5,1.8) {$1$};
\draw [semithick,postaction={decorate}] (0,.85) -- (1,1.15);
\node at (.5,.8) {$1$};
\draw [semithick,postaction={decorate}] (1,.25) -- (2,.55);
\node at (1.5,.2) {$1$};
\node at (0,-.3) {$k$};
\node at (1,-.3) {$l$};
\node at (2,-.3) {$r$};
\end{tikzpicture}
\;+\;
\begin{tikzpicture}[anchorbase, decoration={markings,mark=at position 0.5 with {\arrow{>}}}]
\draw [semithick,->] (0,0) -- (0,2);
\draw [semithick,->] (1,0) -- (1,2);
\draw [semithick,->] (2,0) -- (2,2);
\draw [semithick,postaction={decorate}] (1,.35) -- (2,.65);
\node at (1.5,.3) {$2$};
\draw [semithick,postaction={decorate}] (0,1.35) -- (1,1.65);
\node at (.5,1.3) {$1$};
\node at (0,-.3) {$k$};
\node at (1,-.3) {$l$};
\node at (2,-.3) {$r$};
\end{tikzpicture}
=0.
\]
\end{definition}

The upward webs defined above naturally assemble in a category that we denote $\iWeb^+$ of universal purely upward-going webs. The reason to introduce this category is that the parameter $n$ (or later $x$) is invisible to this part of the category. We refer to \cite{CKM} (see also~\cite{QS2} or~\cite{TVW}) for related categories and their properties.

\begin{definition} \label{def:iWeb+}
The category $\iWeb^+$ has objects, upward arrows on a horizontal line, labelled in $\N^*$. Morphisms are $\Z[q^{\pm 1}]$-linear combinations of upward webs (modulo their relations).
\end{definition}

The category $\iWeb^+$ is braided, with formulas for the braiding as follows:

\begin{equation}
    \label{eq:braiding}
\begin{tikzpicture}[anchorbase]
\draw [semithick,->] (0,0) -- (1,1);
\node at (0,-.2) {\small $k$};
\draw [semithick] (1,0) -- (.6,.4);
\node at (1,-.2) {\small $l$};
\draw [semithick,->] (.4,.6) -- (0,1);
\end{tikzpicture}    
=
(-q)^{kl}\sum_{\substack{a,b\geq 0 \\ b-a=k-l}} (-q)^{b-k}
\begin{tikzpicture}[anchorbase, decoration={
    markings,
    mark=at position 0.5 with {\arrow{>}}}]
\draw [semithick,->] (0,0) -- (0,1);
\node at (0,-.2) {\small $k$};
\draw [semithick,->] (1,0) -- (1,1);
\node at (1,-.2) {\small $l$};
\draw [semithick,postaction={decorate}] (0,.2) -- (1,.3);
\node at (.5,.1) {\tiny $b$};
\draw [semithick,postaction={decorate}] (1,.7) -- (0,.8);
\node at (.5,.9) {\tiny $a$};
\end{tikzpicture}
\;,\quad
\begin{tikzpicture}[anchorbase]
\draw [semithick,->] (1,0) -- (0,1);
\node at (0,-.2) {\small $k$};
\draw [semithick] (0,0) -- (.4,.4);
\node at (1,-.2) {\small $l$};
\draw [semithick,->] (.6,.6) -- (1,1);
\end{tikzpicture}    
=
(-q)^{-kl}\sum_{\substack{a,b\geq 0 \\ b-a=l-k}} (-q)^{l-b}
\begin{tikzpicture}[anchorbase, decoration={
    markings,
    mark=at position 0.5 with {\arrow{>}}}]
\draw [semithick,->] (0,0) -- (0,1);
\node at (0,-.2) {\small $k$};
\draw [semithick,->] (1,0) -- (1,1);
\node at (1,-.2) {\small $l$};
\draw [semithick,postaction={decorate}] (1,.2) -- (0,.3);
\node at (.5,.1) {\tiny $b$};
\draw [semithick,postaction={decorate}] (0,.7) -- (1,.8);
\node at (.5,.9) {\tiny $a$};
\end{tikzpicture}
\end{equation}

\begin{remark}\label{rem:braiding}
If one of the strand is colored $1$, the braiding reduces to 
\[
\begin{tikzpicture}[anchorbase,semithick]
\draw[->] (0.3,0)node[below]{\scriptsize{$k$}} -- (1.3,1);
\draw (1.3,0)node[below]{\scriptsize{$1$}} -- (0.9,0.4);
\draw[->] (0.7,0.6)-- (0.3,1);

\node at(2.5,0.5) {$=~(-q)^{k-1}$};
\draw[->] (3.5,0)node[below]{\scriptsize{$k$}} -- (3.5,1)node[above] {\scriptsize{$1$}};
\draw[->] (4.5,0)node[below]{\scriptsize{$1$}}-- (4.5,1)node[above] {\scriptsize{$k$}};
\draw (3.5,0.4) -- (4.5,0.6) node[midway,sloped,below] {\scriptsize{$k-1$}};

\node at(5.5,0.5) {$+~(-q)^k$};
\draw (6.5,0)node[below] {\scriptsize{$k$}} to[out=90,in=-135] (6.75,0.3);
\draw (7,0)node[below] {\scriptsize{$1$}} to[out=90,in=-45] (6.75,0.3);
\draw (6.75,0.3)--(6.75,0.7);
\draw[->] (6.75,0.7)to[out=135,in=-90] (6.5,1)node[above] {\scriptsize{$1$}};
\draw[->] (6.75,0.7)to[out=45,in=-90] (7,1)node[above] {\scriptsize{$k$}};
\end{tikzpicture}
\text{ and }
\begin{tikzpicture}[anchorbase,semithick]
\draw[->] (0.3,0)node[below]{\scriptsize{$1$}} -- (1.3,1);
\draw (1.3,0)node[below]{\scriptsize{$k$}} -- (0.9,0.4);
\draw[->] (0.7,0.6)-- (0.3,1);

\node at(2.5,0.5) {$=~(-q)^{k-1}$};
\draw[->] (3.5,0)node[below]{\scriptsize{$1$}} -- (3.5,1)node[above] {\scriptsize{$k$}};
\draw[->] (4.5,0)node[below]{\scriptsize{$k$}}-- (4.5,1)node[above] {\scriptsize{$1$}};
\draw (3.5,0.6) -- (4.5,0.4) node[midway,sloped,above] {\scriptsize{$k-1$}};

\node at(5.5,0.5) {$+~(-q)^k$};
\draw (6.5,0)node[below] {\scriptsize{$1$}} to[out=90,in=-135] (6.75,0.3);
\draw (7,0)node[below] {\scriptsize{$k$}} to[out=90,in=-45] (6.75,0.3);
\draw (6.75,0.3)--(6.75,0.7);
\draw[->] (6.75,0.7)to[out=135,in=-90] (6.5,1)node[above] {\scriptsize{$k$}};
\draw[->] (6.75,0.7)to[out=45,in=-90] (7,1)node[above] {\scriptsize{$1$}};
\end{tikzpicture}
\]
\noindent In particular, the following relation will be useful. 
\begin{equation*}
\begin{tikzpicture}[anchorbase]
\draw [semithick] (0,0) -- (0,.4);
\draw [semithick] (0,.4) to [out=120,in=-90] (-.3,.5) to [out=90,in=-90] (.3,.9) -- (.3,1);
\draw [semithick] (0,.4) to [out=60,in=-90] (.3,.5) to [out=90,in=-45] (.1,.66);
\draw [semithick] (-.1,.74) to [out=135,in=-90] (-.3,.9) -- (-.3,1);
\node at (0,-.3) {$k$};
\node at (-.3,1.3) {$1$};
\node [rotate=90] at (.3,1.4) {$k-1$};
\end{tikzpicture}
\;=(-1)^{k-1}q^{2k-2}\;
\begin{tikzpicture}[anchorbase]
\draw [semithick] (0,0) -- (0,.4);
\draw [semithick] (0,.4) to [out=120,in=-90] (-.3,.6) -- (-.3,1);
\draw [semithick] (0,.4) to [out=60,in=-90] (.3,.6) -- (.3,1);
\node at (0,-.3) {$k$};
\node at (-.3,1.3) {$1$};
\node [rotate=90] at (.3,1.4) {$k-1$};
\end{tikzpicture}
\end{equation*}

This relation can be proven as follows:
\begin{gather*}
    \begin{tikzpicture}[anchorbase]
\draw [semithick] (0,0) -- (0,.4);
\draw [semithick] (0,.4) to [out=120,in=-90] (-.3,.5) to [out=90,in=-90] (.3,.9) -- (.3,1);
\draw [semithick] (0,.4) to [out=60,in=-90] (.3,.5) to [out=90,in=-45] (.1,.66);
\draw [semithick] (-.1,.74) to [out=135,in=-90] (-.3,.9) -- (-.3,1);
\node at (0,-.3) {$k$};
\node at (-.3,1.3) {$1$};
\node [rotate=90] at (.3,1.4) {$k-1$};
\end{tikzpicture}
=(-q)^{k-2}
    \begin{tikzpicture}[anchorbase]
\draw [semithick] (0,0) -- (0,.4);
\draw [semithick] (0,.4) to [out=120,in=-90] (-.3,.5);
\draw [semithick] (-.3,.5) to [out=90,in=-90]  (-.3,1);
\draw [semithick] (0,.4) to [out=60,in=-90] (.3,.5);
\draw [semithick] (.3,.5) to [out=90,in=-90] (.3,1);
\draw [semithick] (-.3,.65) to [out=60,in=-120] (.3,.85);
\node at (0,-.3) {$k$};
\node at (-.3,1.3) {$1$};
\node at (.4,.4) {$1$};
\node [rotate=90] at (.3,1.4) {$k-1$};
\end{tikzpicture}
+(-q)^{k-1}
    \begin{tikzpicture}[anchorbase]
\draw [semithick] (0,0) -- (0,.4);
\draw [semithick] (0,.4) to [out=120,in=-90] (-.3,.5) to [out=90,in=-120] (0,.6);
\draw [semithick] (0,.4) to [out=60,in=-90]  (.3,.5) to [out=90,in=-60] (0,.6);
\draw [semithick] (0,.6) -- (0,.7);
\draw [semithick] (0,.7) to [out=120,in=-90] (-.3,.85) -- (-.3,1);
\draw [semithick] (0,.7) to [out=60,in=-90] (.3,.85) -- (.3,1);
\node at (0,-.3) {$k$};
\node at (-.3,1.3) {$1$};
\node at (.4,.4) {$1$};
\node [rotate=90] at (.3,1.4) {$k-1$};
\end{tikzpicture}
=(-q)^{k-2}
    \begin{tikzpicture}[anchorbase]
\draw [semithick] (0,0) -- (0,.4);
\draw [semithick] (0,.4) to [out=120,in=-90] (-.3,.5);
\draw [semithick] (-.3,.5) to [out=90,in=-90]  (-.3,1);
\draw [semithick] (0,.4) to [out=60,in=-90] (.3,.5);
\draw [semithick] (.3,.5) to [out=90,in=-90] (.3,1);
\draw [semithick] (.3,.65) to [out=120,in=-90] (0,.75) to [out=90,in=-120] (.3,.85);
\node at (0,-.3) {$k$};
\node at (-.3,1.3) {$1$};
\node at (.4,.75) {$1$};
\node [rotate=90] at (.3,1.4) {$k-1$};
\end{tikzpicture}
+(-1)^{k-1}\qn{k}
\begin{tikzpicture}[anchorbase]
\draw [semithick] (0,0) -- (0,.4);
\draw [semithick] (0,.4) to [out=120,in=-90] (-.3,.6) -- (-.3,1);
\draw [semithick] (0,.4) to [out=60,in=-90] (.3,.6) -- (.3,1);
\node at (0,-.3) {$k$};
\node at (-.3,1.3) {$1$};
\node [rotate=90] at (.3,1.4) {$k-1$};
\end{tikzpicture}
\\
=((-1)^{k-2}\qn{k-1}+(-1)^{k-1}\qn{k})
\begin{tikzpicture}[anchorbase]
\draw [semithick] (0,0) -- (0,.4);
\draw [semithick] (0,.4) to [out=120,in=-90] (-.3,.6) -- (-.3,1);
\draw [semithick] (0,.4) to [out=60,in=-90] (.3,.6) -- (.3,1);
\node at (0,-.3) {$k$};
\node at (-.3,1.3) {$1$};
\node [rotate=90] at (.3,1.4) {$k-1$};
\end{tikzpicture}
=(-1)^{k-1}q^{2k-2}
\begin{tikzpicture}[anchorbase]
\draw [semithick] (0,0) -- (0,.4);
\draw [semithick] (0,.4) to [out=120,in=-90] (-.3,.6) -- (-.3,1);
\draw [semithick] (0,.4) to [out=60,in=-90] (.3,.6) -- (.3,1);
\node at (0,-.3) {$k$};
\node at (-.3,1.3) {$1$};
\node [rotate=90] at (.3,1.4) {$k-1$};
\end{tikzpicture}.
\end{gather*}
\end{remark}

\subsection{A universal web category, and the HOMFLY-PT invariant}

In this section, we want to define a version of the category of webs that allows for up and down arrows, and into which links will naturally map. The definitions we present mimic the use of $q$-rationals from Section~\ref{sec:qrat}. For this reason, the coefficients that appear in our definitions might look peculiar to readers used to the more symmetric definition of quantum integers $[n]=\frac{q^n-q^{-n}}{q-q^{-1}}$. It is nonetheless an easy exercise to check that the two versions are equivalent.

To allow edges to go up and down expresses on the category side by saying that we want to consider a pivotal version of the category, which allows for duals.

Here we will follow~\cite[Definition 6.5]{QS2}.
\begin{definition} \label{def:iWeb}
We define $\iWeb$ to be the braided, pivotal category generated by $\iWeb^+$, over the ring $\Z[q^{\pm 1},a^{\pm 1}][(q-q^{-1})^{-1}]$, with additional relations:
\begin{itemize}
    \item circle evaluation:
    \begin{equation}
    \label{eq:unknot_i}
\begin{tikzpicture}[anchorbase]
\draw [semithick, ->] (0,0) arc (0:360:.5) ;
\node at (.2,-.3) {$k$};
\end{tikzpicture}
\; =\;
\begin{tikzpicture}[anchorbase]
\draw [semithick, ->] (0,0) arc (0:-360:.5) ;
\node at (.2,-.3) {$k$};
\end{tikzpicture}
\; = \;
\mu_k.
\end{equation}
The $\mu_k$'s can be explicitly determined, as in~\cite{QS2}. For the purpose of the definition, only the $k=1$ case is needed, as the other ones follow. Let us thus just indicate that $\mu_1=\frac{a-a^{-1}}{q-q^{-1}}$.\footnote{To be complete, $\mu_k=\frac{a-a^{-1}}{q-q^{-1}}\frac{aq^{-1}-a^{-1}q}{q-q^{-1}}\cdots \frac{aq^{1-k}-a^{-1}q^{k-1}}{q-q^{-1}}\frac{q^{\frac{k(k-1)}{2}}}{\qn{k}!}$.}
\item sideway digon removal:
\begin{equation}
    \label{eq:sideway_digon_1}
    \begin{tikzpicture}[anchorbase]
    \draw [semithick,->] (0,0) -- (0,3) ;
    \draw [semithick,->] (0,2) to [out=90,in=90] (1,2) -- (1,1.5);
    \draw [semithick] (1,1.5) -- (1,1) to [out=-90,in=-90] (0,1);
    \node at (0,-.3) {$1$};
    \node at (0,3.3) {$1$};
    \node at (1.3,1.5) {$1$};
    \node at (-.3,1.5) {$2$}; 
    \end{tikzpicture}
    \; = \;
    \begin{tikzpicture}[anchorbase]
    \draw [semithick,->] (0,0) -- (0,3) ;
    \draw [semithick,->] (0,2) to [out=90,in=90] (-1,2) -- (-1,1.5);
    \draw [semithick] (-1,1.5) -- (-1,1) to [out=-90,in=-90] (0,1);
    \node at (0,-.3) {$1$};
    \node at (0,3.3) {$1$};
    \node at (-1.3,1.5) {$1$};
    \node at (.3,1.5) {$2$}; 
    \end{tikzpicture}
    \; = \;
\frac{aq^{-1}-a^{-1}q}{q-q^{-1}}.
    \begin{tikzpicture}[anchorbase]
    \draw [semithick,->] (0,0) -- (0,3) ;
    \node at (0,-.3) {$1$};
    \node at (0,3.3) {$1$};
    \end{tikzpicture}
\end{equation}
\item square removal (together with its mirror image):
\begin{equation}
    \label{eq:mixed_square_QS}
    \begin{tikzpicture}[anchorbase,decoration={
    markings,
    mark=at position 0.5 with {\arrow{>}}}]
    \draw[semithick, postaction={decorate}] (0,0) to [out=45,in=180] (.5,.5);
    \draw [semithick ,postaction={decorate}] (.5,.5) -- (1.5,.5);
    \draw [semithick,postaction={decorate}] (.5,1.5) to [out=180,in=180] (.5,.5);
    \draw [semithick, postaction={decorate}] (1.5,1.5) -- (.5,1.5);
    \draw [semithick,postaction={decorate}] (1.5,.5) to [out=0,in=0] (1.5,1.5);
    \draw [semithick,->] (1.5,.5) to [out=0,in=135] (2,0);
    \draw [semithick, postaction={decorate}] (2,2) to [out=-135,in=0] (1.5,1.5);
    \draw [semithick, ->] (.5,1.5) to [out=185,in=-45] (0,2);
    \node at (-.2,-.2) {$1$};
    \node at (-.2,2.2) {$1$};
    \node at (2.2,-.2) {$1$};
    \node at (2.2,2.2) {$1$};
    \node at (1,.2) {$2$};
    \node at (1,1.8) {$2$};
    \end{tikzpicture}
    \; = \;
    \begin{tikzpicture}[anchorbase]
    \draw [semithick, ->] (0,0) to [out=45,in=-45] (0,2);
    \draw [semithick,->] (2,2) to [out=-135,in=135] (2,0);
    \end{tikzpicture}
    \; + \;
    \frac{aq^{-2}-a^{-1}q^2}{q-q^{-1}}\;
    \begin{tikzpicture}[anchorbase]
    \draw [semithick, ->] (0,0) to [out=45,in=135] (2,0);
    \draw [semithick,->] (2,2) to [out=-135,in=-45] (0,2);
    \end{tikzpicture}
\end{equation}

\end{itemize}
\end{definition}

\begin{lemma}\label{lemma:curlHomfly}
Colored twists act as follows:
\[
    \begin{tikzpicture}[anchorbase]
    \draw [semithick] (0,1) -- (0,.65);
    \draw [semithick] (.1,.4) to [out=-80,in=180] (.3,.2) to [out=0,in=-90] (.6,.5) to [out=90,in=0] (.3,.8) to [out=180,in=90] (0,.5) -- (0,0);
        \node at (0,-.3) {$k$};
    \node at (0,1.3) {\vphantom{$k$}};
\end{tikzpicture}
\; = a^{-k}q^{k(2k-1)}\;
    \begin{tikzpicture}[anchorbase]
    \draw [semithick] (0,0) -- (0,1);
    \node at (0,-.3) {$k$};
    \node at (0,1.3) {\vphantom{$k$}};
    \end{tikzpicture},\quad
\begin{tikzpicture}[anchorbase]
    \draw [semithick] (0,0) -- (0,.35);
    \draw [semithick] (.1,.6) to [out=80,in=180] (.3,.8) to [out=0,in=90] (.6,.5) to [out=-90,in=0] (.3,.2) to [out=180,in=-90] (0,.5) -- (0,1);
    \node at (0,-.3) {$k$};
    \node at (0,1.3) {\vphantom{$k$}};
    \end{tikzpicture}= a^kq^{-k(2k-1)} \;
    \begin{tikzpicture}[anchorbase]
    \draw [semithick] (0,0) -- (0,1);
    \node at (0,-.3) {$k$};
    \node at (0,1.3) {\vphantom{$k$}};
    \end{tikzpicture}.
\]
\end{lemma}

\begin{proof}
We only prove the first equality, as the second one follows. The proof goes by induction. Denote $\eta_k$ the coefficient we are willing to compute. We start by computing the base case of $k=1$:
\[
    \begin{tikzpicture}[anchorbase]
    \draw [semithick,<-] (0,2) -- (0,1.15);
    \draw [semithick] (.1,.9) to [out=-80,in=180] (.3,.7) to [out=0,in=-90] (.6,1) to [out=90,in=0] (.3,1.3) to [out=180,in=90] (0,1) -- (0,0);
    \end{tikzpicture}
    =
    \begin{tikzpicture}[anchorbase]
    \draw [semithick,->] (0,0) -- (0,2);
    \draw [semithick,->] (.75,1) arc (360:0:.25);
    \end{tikzpicture}
    -q\;
    \begin{tikzpicture}[anchorbase]
    \draw [semithick] (0,0) to [out=90,in=-120] (.25,.75);
    \draw [semithick,->] (.25,.75) -- (.25,1);
    \node at (0,1) {\tiny $2$};
    \draw [semithick] (.25,1) -- (.25,1.25);
    \draw [semithick,->] (.25,1.25) to [out=120,in=-90] (0,2);
    \draw [semithick] (.25,1.25) to [out=60,in=180] (.5,1.35) to [out=0,in=0] (.5,.65) to [out=180,in=-60] (.25,.75);
    \end{tikzpicture}
    =(\frac{a-a^{-1}}{q-q^{-1}}-q\frac{aq^{-1}-a^{-1}q}{q-q^{-1}})\;
    \begin{tikzpicture}[anchorbase]
    \draw [semithick,->] (0,0) -- (0,2);
    \end{tikzpicture}
\]
Above we have used \cref{eq:unknot_i} and \cref{eq:sideway_digon_1}. It follows that $\eta_1=a^{-1}q$.

In the general case, one can first explode the strand:
\[
    \begin{tikzpicture}[anchorbase]
    \draw [semithick] (0,2) -- (0,1.15);
    \draw [semithick] (.1,.9) to [out=-80,in=180] (.3,.7) to [out=0,in=-90] (.6,1) to [out=90,in=0] (.3,1.3) to [out=180,in=90] (0,1) -- (0,0);
        \node at (0,-.3) {$k$};
    \node at (0,2.3) {\vphantom{$k$}};
\end{tikzpicture}
\;=\frac{1}{q^{1-k}\qn{k}}
\begin{tikzpicture}[anchorbase]
\draw[semithick] (0,0) -- (0,.4);
\draw [semithick] (0,.4) to [out=120,in=180] (.75,1.5) to [out=0,in=0] (.75,.5) to [out=180,in=-40] (.25,.7);
\draw [semithick] (.15,.8) to [out=135,in=-75] (.05,.95);
\draw [semithick] (-.05,1.1) to [out=110,in=-120] (0,1.6);
\draw [opacity=.4] (0,.4) to [out=60,in=180] (.75,1.25) to [out=0,in=0] (.75,.75) to [out=180,in=-60] (0,1.6);
\draw [semithick] (0,.4) to [out=60,in=180] (.75,1.25) to [out=0,in=0] (.75,.75) to [out=180,in=-45] (.35,.95);
\draw [semithick] (.28,1.05) to [out=125,in=-70] (.22,1.17);
\draw [semithick] (.15,1.3) to [out=110,in=-60] (0,1.6);
\draw[semithick] (0,1.6) -- (0,2);
\node at (0,-.3) {$k$};
\node at (0,2.3) {\vphantom{$k$}};
\node [rotate=90] at (-.25,.6) {\tiny $k-1$};
\end{tikzpicture}
\;=\frac{1}{q^{1-k}\qn{k}}\;
\begin{tikzpicture}[anchorbase]
\draw [semithick] (0,0) -- (0,.4);
\draw [semithick] (0,.4) to [out=120,in=-90] (-.3,.5) to [out=90,in=-90] (.3,.9);
\draw [semithick] (0,.4) to [out=60,in=-90] (.3,.5) to [out=90,in=-45] (.1,.66);
\draw [semithick] (-.1,.74) to [out=135,in=-90] (-.3,.9);
\draw [semithick] (-.3,.9) -- (-.3,1) arc(180:-0:.15) arc(0:-160:.13);
\draw [semithick] (.3,.9) -- (.3,1) arc(180:-0:.15) arc(0:-160:.13);
\draw [semithick] (-.3,1.1) to [out=90,in=-90] (.3,1.5) to [out=90,in=-60] (0,1.6);
\draw [semithick] (.3,1.1) to [out=90,in=-45] (.1,1.26);
\draw [semithick] (-.1,1.34) to [out=135,in=-90] (-.3,1.5) to [out=90,in=-120] (0,1.6);
\draw [semithick] (0,1.6) -- (0,2);
\node at (0,-.3) {$k$};
\node at (0,2.3) {\vphantom{$k$}};
\end{tikzpicture}
\;=\frac{1}{q^{1-k}\qn{k}}\eta_1\eta_{k-1}q^{4(k-1)}\;
\begin{tikzpicture}[anchorbase]
\draw [semithick] (0,0) -- (0,.4);
\draw [semithick] (0,.4) to [out=120,in=-90] (-.3,1) to [out=90,in=-120] (0,1.6);
\draw [semithick] (0,.4) to [out=60,in=-90] (.3,1) to [out=90,in=-60] (0,1.6);
\draw [semithick] (0,1.6) -- (0,2);
\node at (0,-.3) {$k$};
\node at (0,2.3) {\vphantom{$k$}};
\node [rotate=-120] at (.25,.6) {\tiny $k-1$};
\end{tikzpicture}
\;=\eta_1\eta_{k-1}q^{4(k-1)}\;
\begin{tikzpicture}[anchorbase]
\draw [semithick] (0,0) -- (0,2);
\node at (0,-.3) {$k$};
\node at (0,2.3) {\vphantom{$k$}};
\end{tikzpicture}
\]
Throughout we have used the relation written in Remark \ref{rem:braiding}. Then the claim follows by extracting the value of the $\eta_k$'s from the recursion.
\end{proof}

\begin{lemma}
Higher versions of \cref{eq:sideway_digon_1} can be deduced, yielding:
\begin{equation}
    \label{eq:sideway_digon_k}
    \begin{tikzpicture}[anchorbase]
    \draw [semithick,->] (0,0) -- (0,3) ;
    \draw [semithick,->] (0,2) to [out=90,in=90] (1,2) -- (1,1.5);
    \draw [semithick] (1,1.5) -- (1,1) to [out=-90,in=-90] (0,1);
    \node at (0,-.3) {$k$};
    \node at (0,3.3) {$k$};
    \node at (1.3,1.5) {$1$};
    \node[rotate=90] at (-.3,1.5) {$k+1$}; 
    \end{tikzpicture}
    \; = \;
    \begin{tikzpicture}[anchorbase]
    \draw [semithick,->] (0,0) -- (0,3) ;
    \draw [semithick,->] (0,2) to [out=90,in=90] (-1,2) -- (-1,1.5);
    \draw [semithick] (-1,1.5) -- (-1,1) to [out=-90,in=-90] (0,1);
    \node at (0,-.3) {$k$};
    \node at (0,3.3) {$k$};
    \node at (-1.3,1.5) {$1$};
    \node[rotate=90] at (.3,1.5) {$k+1$}; 
    \end{tikzpicture}
    \; = \;
\frac{aq^{-k}-a^{-1}q^k}{q-q^{-1}}.
    \begin{tikzpicture}[anchorbase]
    \draw [semithick,->] (0,0) -- (0,3) ;
    \node at (0,-.3) {$k$};
    \node at (0,3.3) {$k$};
    \end{tikzpicture}
\end{equation}
\end{lemma}

\begin{proof}
We prove the equality for the right digon removal, the left case is symmetric. Applying \cref{eq:updigon} and \cref{eq:triwebs}, we get
\[
\begin{tikzpicture}[anchorbase,semithick,decoration={
    markings,
    mark=at position 0.5 with {\scriptsize{\arrow{>}}}}]
\begin{scope}
\draw[->] (0,0)node[below] {\scriptsize{$k$}} --(0,2);
\draw[->] (0,1.2) to[out=90,in=90] (0.5,1.2) -- (0.5,0.8) node[midway,right] {\scriptsize{$1$}};
\draw (0.5,0.8) to[out=-90,in=-90] (0,0.8);

\node at(1.3,1) {$=$};
\end{scope}
\begin{scope}[shift={(2.7,0)}]
\node at(-0.6,1) {$\dfrac{q^{k-1}}{\qn{k}}$};
\draw (0,0)node[below] {\scriptsize{$k$}}--(0,0.2);
\draw (0,0.2)to[out=40,in=-40] (0,0.6);
\draw (0,0.2)to[out=140,in=220] (0,0.6);
\node at(0.5,0.4) {\scriptsize{$k-1$}};
\node at(-0.2,0.4) {\scriptsize{$1$}};
\draw[->] (0,0.6)--(0,2);
\draw[->] (0,1.5) to[out=90,in=90] (0.5,1.5) -- (0.5,1.1) node[midway,right] {\scriptsize{$1$}};
\draw (0.5,1.1) to[out=-90,in=-90] (0,1.1);
\node at(1.3,1) {$=$};
\end{scope}
\begin{scope}[shift={(6,0)}]
\node at(-1,1) {$\dfrac{q^{k-1}}{\qn{k}}$};
\draw[->] (0,0)node[below] {\scriptsize{$k$}}--(0,2);
\draw[->] (0,1.5) to[out=90,in=90] (0.5,1.5) -- (0.5,0.6) node[midway,right] {\scriptsize{$1$}};
\draw (0.5,0.6) to[out=-90,in=-90] (0,0.6);
\draw[->] (0,0.2) to[out=90,in=-90] (-0.5,0.4)node[below,left] {\scriptsize{$1$}} -- (-0.5,1);
\draw (-0.5,1) to[out=90,in=-90] (0,1.3);
\end{scope}
\end{tikzpicture}
\]
\noindent Then we can apply the braiding relation to the upper part of this last web (see Remark \ref{rem:braiding}).
\[
\begin{tikzpicture}[anchorbase,semithick,decoration={
    markings,
    mark=at position 0.5 with {\scriptsize{\arrow{>}}}}]
\begin{scope}
\draw[->] (0,0)node[below] {\scriptsize{$k$}} --(0,2);
\draw[->] (0,1.2) to[out=90,in=90] (0.5,1.2) -- (0.5,0.8) node[midway,right] {\scriptsize{$1$}};
\draw (0.5,0.8) to[out=-90,in=-90] (0,0.8);
\node at(1.3,1) {$=$};
\end{scope}
\begin{scope}[shift={(3.2,0)}]
\node at(-1.2,1) {$\dfrac{(-1)^k}{q\qn{k}}$};
\draw (0,0)node[below] {\scriptsize{$k$}}--(0,1.6);
\draw[->] (0,1.8) -- (0,2);
\draw[->] (0,0.2) to[out=90,in=-90] (-0.5,0.4) -- (-0.5,1.4);
\draw (-0.5,1.4) to[out=90,in=90] (0.5,1.4);
\draw[->] (0.5,1.4)--(0.5,0.6)node[midway,right] {\scriptsize{$1$}};
\draw (0.5,0.6) to[out=-90,in=-90] (0,0.6);
\node at(1.4,1) {$+$};
\end{scope}
\begin{scope}[shift={(6.5,0)}]
\node at(-1,1) {$\dfrac{q^{k-2}}{\qn{k}}$};
\draw[->] (0,0)node[below] {\scriptsize{$k$}}--(0,2);
\draw[->] (0,1.4) to[out=90,in=90] (0.5,1.4) -- (0.5,0.6) node[midway,right] {\scriptsize{$1$}};
\draw (0.5,0.6) to[out=-90,in=-90] (0,0.6);
\draw[->] (0,0.2) to[out=90,in=-90] (-0.5,0.4) -- (-0.5,1.6)node[midway,right] {\scriptsize{$1$}};
\draw (-0.5,1.6) to[out=90,in=-90] (0,1.8);
\end{scope}
\end{tikzpicture}
\]
\noindent By induction hypothesis and \cref{eq:updigon}, the term on the right of the sum is equal to $q^{-1}\frac{aq^{-(k-1)} - a^{-1}q^{k-1}}{q-q^{-1}}$. For the first term of the sum, we can use the relation of Remark \ref{rem:braiding}, and then a curl removal. 
\[
\begin{tikzpicture}[anchorbase,semithick,decoration={
    markings,
    mark=at position 0.5 with {\scriptsize{\arrow{>}}}}]
\begin{scope}
\draw (0,0)node[below] {\scriptsize{$k$}}--(0,1.6);
\draw[->] (0,1.8) -- (0,2);
\draw[->] (0,0.2) to[out=90,in=-90] (-0.5,0.4) -- (-0.5,1.4);
\draw (-0.5,1.4) to[out=90,in=90] (0.5,1.4);
\draw[->] (0.5,1.4)--(0.5,0.6)node[midway,right] {\scriptsize{$1$}};
\draw (0.5,0.6) to[out=-90,in=-90] (0,0.6);
\node at(1.4,1) {$=$};
\end{scope}
\begin{scope}[shift={(3,0)}]
\draw (0,0)node[below] {\scriptsize{$k$}} -- (0,1);
\draw[->] (0,1.2)--(0,2);
\draw[->] (0,0.2) to[out=90,in=-90] (-0.5,0.4) -- (-0.5,0.8) node[midway,left] {\scriptsize{$1$}};
\draw (-0.5,0.8) to[out=90,in=90] (0.6,0.8);
\draw (0.6,0.8) to[out=-90,in=-45] (0.28,0.95);
\draw[postaction={decorate}] (0.2,1.2) -- (0,1.8);
\node at(1.3,1) {$=$};
\end{scope}
\begin{scope}[shift={(7.5,0)}]
\node at(-1.6,1) {$(-1)^{k-1}q^{2k-2}$};
\draw (0,0)node[below] {\scriptsize{$k$}} -- (0,0.4);
\draw[postaction={decorate}] (0,0.4) to[out=135,in=-90] (0,1.8);
\node at(-0.5,0.5) {\scriptsize{$k-1$}};
\draw[->] (0,1.8)--(0,2);
\draw (0,0.4) to[out=45,in=-90] (0.4,1.2);
\draw (0.4,1.2) to[out=90,in=90] (0.8,1.2);
\draw (0.8,1.2) to[out=-90,in=-45] (0.5,0.9);
\draw (0.3,1.1) --(0,1.8);
\node at(0.25,0.5) {\scriptsize{$1$}};
\node at(1.3,1) {$=$};
\end{scope}
\begin{scope}[shift={(11,0)}]
\node at(0,1) {$(-1)^{k-1}q^{k}a^{-1}\qn{k}$};
\draw[->] (2,0)node[below] {\scriptsize{$k$}} -- (2,2);
\end{scope}
\end{tikzpicture}
\]
\noindent From this, we finally get the result, as:
\[
-q^{k-1}a^{-1}+q^{-1}\frac{aq^{-(k-1)}-a^{-1}q^{k-1}}{q-q^{-1}}=\frac{aq^{-k}-a^{-1}q^k}{q-q^{-1}}.
\]
\end{proof}

The category above is denoted~$\mathrm{\bf Sp}(\beta)$ in~\cite{QS2}, where it is proven that it does not degenerate (this is a consequence of \cite[Proposition 6.7]{QS2}), and it is studied in details in~\cite{Brundan}. This category is also closely related to Vogel's universal category from~\cite{Vogel}, where the loop is kept as a formal parameter and where non-degeneracy is proved in Proposition 1.6.

A classical fact is that $\End_{\iWeb}(\emptyset)\simeq \Z[q^{\pm 1},a^{\pm 1}][(q-q^{-1})^{-1}]$, which allows to define a (framed) knot invariant by considering a knot diagram into this endomorphism space.

\begin{definition}
  Given $K$ a framed knot presented by a diagram $D$, the value:
  \[
  [K]_a:=[D]_a\in \End_{\iWeb}(\emptyset)\simeq \Z[q^{\pm 1},a^{\pm 1}][(q-q^{-1})^{-1}]
  \]
  is the HOMFLY-PT polynomial\footnote{Strictly speaking, this is not a polynomial but rather a rational fraction.} of $K$.
\end{definition}

Figure~\ref{fig:HOMFLY-PT} shows more direct computation rules for the HOMFLY-PT polynomial, which follow from the definition.
\begin{figure}[H]
\fbox{
\begin{minipage}{.45\textwidth}
\begin{center}
    {\bf HOMFLY-PT}
\end{center}
\begin{itemize}
    \item unknot evaluation:
    \[\begin{tikzpicture}[anchorbase] \draw [semithick] (0,0) circle (.5); \end{tikzpicture}=\mu;\]
    \item Conway skein relation:
    \[q^{-1}\begin{tikzpicture}[anchorbase] \draw [semithick,->] (0,0) -- (1,1); \draw [semithick] (1,0) -- (.6,.4); \draw [semithick,->] (.4,.6) -- (0,1); \end{tikzpicture}- q\begin{tikzpicture}[anchorbase] \draw [semithick,->] (1,0) -- (0,1); \draw [semithick] (0,0) -- (.4,.4); \draw [semithick,->] (.6,.6) -- (1,1); \end{tikzpicture}=(q^{-1}-q) \begin{tikzpicture}[anchorbase] \draw [semithick, ->] (0,0) to [out=45,in=-45] (0,1); \draw [semithick, ->] (1,0) to [out=135,in=-135] (1,1); \end{tikzpicture};
    \]
    \item curls: 
    \[
    \begin{tikzpicture}[anchorbase] \draw [semithick] (0,0) -- (0,.35); \draw [semithick] (.1,.6) to [out=80,in=180] (.3,.8) to [out=0,in=90] (.6,.5) to [out=-90,in=0] (.3,.2) to [out=180,in=-90] (0,.5) -- (0,1); \end{tikzpicture}=q^{-1}a\;\begin{tikzpicture}[anchorbase] \draw [semithick] (0,0) -- (0,1); \end{tikzpicture},\quad \begin{tikzpicture}[anchorbase] \draw [semithick] (0,1) -- (0,.65); \draw [semithick] (.1,.4) to [out=-80,in=180] (.3,.2) to [out=0,in=-90] (.6,.5) to [out=90,in=0] (.3,.8) to [out=180,in=90] (0,.5) -- (0,0); \end{tikzpicture}=qa^{-1}\;\begin{tikzpicture}[anchorbase] \draw [semithick] (0,0) -- (0,1);
    \end{tikzpicture};
    \]
    \item relation between variables: 
    \[
    (q-q^{-1})\mu=a-a^{-1}.
    \]
\end{itemize}
\end{minipage}
}
\caption{Computing the HOMFLY-PT polynomial}
\label{fig:HOMFLY-PT}
\end{figure}
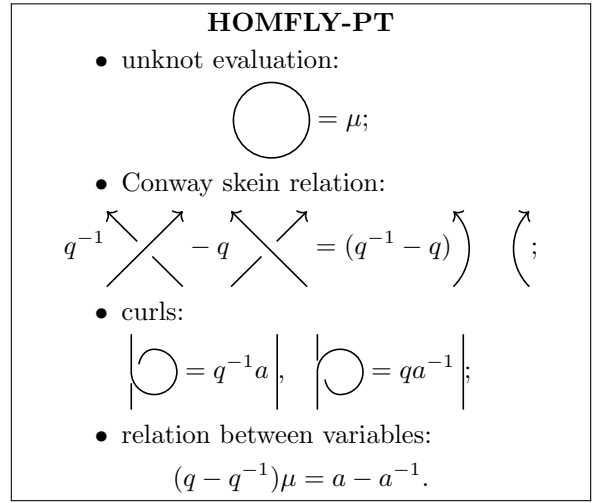

\subsection{Specializations and Reshetikhin-Turaev polynomials}

A key feature of the HOMFLY-PT polynomial that was already mentioned is that it specializes to the Reshetikhin-Turaev $U_q(\mathfrak{gl}_n)$ invariants. This also has a version in terms of webs. Let us first define the relevant categories.

\begin{definition}\label{def:nWeb}
  We define $\nWeb^{\leq n}$ to be the braided, pivotal category over the ring $\Z[q^{\pm 1}]$ generated by $\iWeb^+$, with additional relations:
  \begin{itemize}
    \item circle evaluation:
    \begin{equation}
    \label{eq:unknot_gln}
\begin{tikzpicture}[anchorbase]
\draw [semithick, ->] (0,0) arc (0:360:.5) ;
\node at (.2,-.3) {$k$};
\end{tikzpicture}
\; =\;
\begin{tikzpicture}[anchorbase]
\draw [semithick, ->] (0,0) arc (0:-360:.5) ;
\node at (.2,-.3) {$k$};
\end{tikzpicture}
\; = \;
q^{k(n-k)}\qbin{k}{n}.
\end{equation}
\item sideway digon removal:
\begin{equation}
    \label{eq:sideway_digon_1_n}
    \begin{tikzpicture}[anchorbase]
    \draw [semithick,->] (0,0) -- (0,3) ;
    \draw [semithick,->] (0,2) to [out=90,in=90] (1,2) -- (1,1.5);
    \draw [semithick] (1,1.5) -- (1,1) to [out=-90,in=-90] (0,1);
    \node at (0,-.3) {$1$};
    \node at (0,3.3) {$1$};
    \node at (1.3,1.5) {$1$};
    \node at (-.3,1.5) {$2$}; 
    \end{tikzpicture}
    \; = \;
    \begin{tikzpicture}[anchorbase]
    \draw [semithick,->] (0,0) -- (0,3) ;
    \draw [semithick,->] (0,2) to [out=90,in=90] (-1,2) -- (-1,1.5);
    \draw [semithick] (-1,1.5) -- (-1,1) to [out=-90,in=-90] (0,1);
    \node at (0,-.3) {$1$};
    \node at (0,3.3) {$1$};
    \node at (-1.3,1.5) {$1$};
    \node at (.3,1.5) {$2$}; 
    \end{tikzpicture}
    \; = \;
q^{2-n}\qn{n-1}.
    \begin{tikzpicture}[anchorbase]
    \draw [semithick,->] (0,0) -- (0,3) ;
    \node at (0,-.3) {$1$};
    \node at (0,3.3) {$1$};
    \end{tikzpicture}
\end{equation}
\item square removal (together with its mirror image):
\begin{equation}
    \label{eq:mixed_square_QS_n}
    \begin{tikzpicture}[anchorbase,decoration={
    markings,
    mark=at position 0.5 with {\arrow{>}}}]
    \draw[semithick, postaction={decorate}] (0,0) to [out=45,in=180] (.5,.5);
    \draw [semithick ,postaction={decorate}] (.5,.5) -- (1.5,.5);
    \draw [semithick,postaction={decorate}] (.5,1.5) to [out=180,in=180] (.5,.5);
    \draw [semithick, postaction={decorate}] (1.5,1.5) -- (.5,1.5);
    \draw [semithick,postaction={decorate}] (1.5,.5) to [out=0,in=0] (1.5,1.5);
    \draw [semithick,->] (1.5,.5) to [out=0,in=135] (2,0);
    \draw [semithick, postaction={decorate}] (2,2) to [out=-135,in=0] (1.5,1.5);
    \draw [semithick, ->] (.5,1.5) to [out=185,in=-45] (0,2);
    \node at (-.2,-.2) {$1$};
    \node at (-.2,2.2) {$1$};
    \node at (2.2,-.2) {$1$};
    \node at (2.2,2.2) {$1$};
    \node at (1,.2) {$2$};
    \node at (1,1.8) {$2$};
    \end{tikzpicture}
    \; = \;
    \begin{tikzpicture}[anchorbase]
    \draw [semithick, ->] (0,0) to [out=45,in=-45] (0,2);
    \draw [semithick,->] (2,2) to [out=-135,in=135] (2,0);
    \end{tikzpicture}
    \; + \;
    q^{3-n}\qn{n-2}\;
    \begin{tikzpicture}[anchorbase]
    \draw [semithick, ->] (0,0) to [out=45,in=135] (2,0);
    \draw [semithick,->] (2,2) to [out=-135,in=-45] (0,2);
    \end{tikzpicture}
\end{equation}
\item
  \begin{equation}
    \begin{tikzpicture}[anchorbase]
      \draw [semithick,->] (0,0) -- (0,1);
      \node at (0,-.3) {$k>n$};
      \node at (0,.3) {\vphantom{$k>n$}};
    \end{tikzpicture}
    =0.
    \label{eq:nbounded}
  \end{equation}
  \end{itemize}

  The knot invariant $[K]_n$ associated to this category is the so-called Reshetikhin-Turaev $U_q(\mathfrak{gl}_n)$ polynomial.
\end{definition}

It is a classical computation to check the following.

\begin{lemma}\label{lem:nWeb}
  The category $\nWeb^{\leq n}$ is a quotient by~\cref{eq:nbounded} of the $a=q^n$ specialization of the category $\iWeb$, and the same holds true for the associated invariants:
  \[
  [K]_n=[K]_a(a=q^n,q).
  \]
\end{lemma}

Interestingly, one can also define the category $\nWeb$ as the specialization of $\iWeb$ without imposing~\cref{eq:nbounded}. This category is the inverse limit of the $U_q(\mathfrak{gl}_{k+n|k})$ representation categories, and controls the same knot invariant (but richer tangle invariants).

Usually, the quantum integers appear as the value of the $\mathfrak{gl}_n$ invariant for the unknot. Here, because of our modified conventions for $q$-integers, the $q$-integers appear as the invariant for the stabilized unknot, as shown below.

\begin{lemma} \label{lem:qint_sunknot}
We have the following equality:
\[
\left[
\begin{tikzpicture}[anchorbase]
\draw[thick] (.1,.1) to [out=45,in=90] (.8,0) to [out=-90,in=-45] (0,0) to [out=135,in=90] (-.8,0) to [out=-90,in=-135] (-.1,-.1);
\end{tikzpicture}
\right]_n=\qn{n}
\]
\end{lemma}
\begin{proof}
  We will first compute the value of the HOMFLY-PT polynomial for this stabilized unknot, before specializing it. If follows from Figure~\ref{fig:HOMFLY-PT} that:
  \[
  \left[
\begin{tikzpicture}[anchorbase]
\draw[thick] (.1,.1) to [out=45,in=90] (.8,0) to [out=-90,in=-45] (0,0) to [out=135,in=90] (-.8,0) to [out=-90,in=-135] (-.1,-.1);
\end{tikzpicture}
\right]_a=\frac{a^2-1}{q^2-1}.
  \]
  Now, replacing $a^2=q^{2n}$ in the fraction above, we get:
  \[
  \frac{q^{2n}-1}{q^2-1}=\qn{n}.
  \]
  \end{proof}

\section{Interpolating web categories} \label{sec:interpolating}

With these categories in place, we are interested in interpolating the natural functors
\[\iWeb~\rightarrow~\nWeb^{\leq n}.\]
We will prove in Proposition~\ref{prop:finite_quotients} that~\cref{eq:nbounded} has no analogues in the case of non-integral parameters, and we are rather led to interpolate the functors $\iWeb\rightarrow \nWeb$.

\subsection{Definition and presentation}

Before proceeding to the definition, we need to introduce some notation. In the usual $\mathfrak{gl}_n$ case, when trying to handle knot invariance, one often has to rescale polynomials by powers of $q$. It turns out that in our more general setting, the rescalings cease to be just monomials, and we recall from~\cref{not:delta_x} the quantity:
\[
\delta_x=\qn{x}-\qn{x-1}.
\]
Although not strictly necessary, it will also be convenient to extract a square root of (the inverse of) $\delta_x$, and we will thus work over the quadratic extension $\Z(q)[\nu_x]$ with:
\begin{equation}
    \label{eq:nux}
    \nu_x^2=\delta_x^{-1} .
\end{equation}

Notice that when $x=n$, $\delta_x=q^{2n-2}$ and $\nu_x=q^{1-n}$.

\begin{definition} \label{def:xWeb}
We define $\xWeb$ to be the $\Z(q)[\nu_x]$-linear category obtained by setting $a=q\nu_x\delta_x$ in $\iWeb$.
\end{definition}

The specialization has been chosen purposefully, as illustrated below.
\begin{lemma} \label{lem:qrat_sunknot}
We have the following equality:
\[
\left[
\begin{tikzpicture}[anchorbase]
\draw[thick] (.1,.1) to [out=45,in=90] (.8,0) to [out=-90,in=-45] (0,0) to [out=135,in=90] (-.8,0) to [out=-90,in=-135] (-.1,-.1);
\end{tikzpicture}
\right]_x=\qn{x}.
\]
\end{lemma}
\begin{proof}
Just as in the proof of \cref{lem:qint_sunknot}, we start from the value of the HOMFLY-PT polynomial for the stabilized unknot.
  \[
  \left[
\begin{tikzpicture}[anchorbase]
\draw[thick] (.1,.1) to [out=45,in=90] (.8,0) to [out=-90,in=-45] (0,0) to [out=135,in=90] (-.8,0) to [out=-90,in=-135] (-.1,-.1);
\end{tikzpicture}
\right]_a=\frac{a^2-1}{q^2-1}.
  \]
  Now, replacing $a^2=(q\nu_x\delta_x)^2=q^2\nu_x^2\delta_x^2=q^2\delta_x$ in the fraction above, we get, using Equation~\ref{shift_equivariance}:
  \[
  \frac{q^2\delta_x-1}{q^2-1}=\frac{q^2\qn{x}-q^2\qn{x-1}-1}{q^2-1}=\frac{q^2\qn{x}-\qn{x}}{q^2-1}=\qn{x}.
  \]
  \end{proof}

More generally, the value of colored stabilized unknots will be of interest. We start by computing the effect of a twist.
\begin{lemma}
Colored twists act as follows:
\[
    \begin{tikzpicture}[anchorbase]
    \draw [semithick] (0,1) -- (0,.65);
    \draw [semithick] (.1,.4) to [out=-80,in=180] (.3,.2) to [out=0,in=-90] (.6,.5) to [out=90,in=0] (.3,.8) to [out=180,in=90] (0,.5) -- (0,0);
        \node at (0,-.3) {$k$};
    \node at (0,1.3) {\vphantom{$k$}};
\end{tikzpicture}
=\nu_x^kq^{2k(k-1)}\;
    \begin{tikzpicture}[anchorbase]
    \draw [semithick] (0,0) -- (0,1);
    \node at (0,-.3) {$k$};
    \node at (0,1.3) {\vphantom{$k$}};
    \end{tikzpicture},\quad
\begin{tikzpicture}[anchorbase]
    \draw [semithick] (0,0) -- (0,.35);
    \draw [semithick] (.1,.6) to [out=80,in=180] (.3,.8) to [out=0,in=90] (.6,.5) to [out=-90,in=0] (.3,.2) to [out=180,in=-90] (0,.5) -- (0,1);
    \node at (0,-.3) {$k$};
    \node at (0,1.3) {\vphantom{$k$}};
    \end{tikzpicture}=\nu_x^{-k}q^{-2k(k-1)}\;
    \begin{tikzpicture}[anchorbase]
    \draw [semithick] (0,0) -- (0,1);
    \node at (0,-.3) {$k$};
    \node at (0,1.3) {\vphantom{$k$}};
    \end{tikzpicture}.
\]
\end{lemma}
\begin{proof}
This is a direct consequence of the specialization $a = \nu_x\delta_x$ in \cref{lemma:curlHomfly}.
\end{proof}
The same proof as the one we wrote for \cref{lem:qint_sunknot} can be used to make binomial coefficients appear as follows.
\begin{lemma}\label{lem:qrat_sunknot_col}
We have the following equality:
\[
\left[
\begin{tikzpicture}[anchorbase]
\draw[thick] (.1,.1) to [out=45,in=90] (.8,0) to [out=-90,in=-45] (0,0) to [out=135,in=90] (-.8,0) to [out=-90,in=-135] (-.1,-.1);
\node at (1,-.2) {$k$};
\end{tikzpicture}
\right]_x=q^{k(k-1)}\qbin{k}{x}.
\]
\end{lemma}

One can show that the $q$-rationals behave very well with respect to this specialization, and extract from Definition~\ref{def:xWeb} a presentation that is very close in flavor to the presentations for $\nWeb$. We will do so in Proposition~\ref{prop:xWeb_presentation}, but before this we start with a technical computation.

\begin{lemma} \label{lem:formulaqbin}
For $r\in \Z$, specializing $a=\nu_x\delta_x q$ yields:
\[
\left(\frac{aq^{-r}-a^{-1}q^r}{q-q^{-1}}\right)_{a=\nu_x\delta_xq}=
q^r\nu_x\qn{x-r}.
\]
\end{lemma}
\begin{proof}
We start with $r\geq 0$. Then:
\[
\frac{aq^{-r}-a^{-1}q^r}{q-q^{-1}}=q^{-1}\frac{aq^{1-r}-a^{-1}q^{r-1}}{q-q^{-1}}-a^{-1}q^{r-1}.
\]
By induction, this becomes after specialization:
\[
q^{r-2}\nu_x\qn{x-r+1}-\nu_xq^{r-2}=q^{r-2}\nu_x(\qn{x-r+1}-1)=q^{r}\nu_x\qn{x-r}.
\]
Now, if $r<0$, we rather use:
\[
\frac{aq^{-r}-a^{-1}q^r}{q-q^{-1}}=q\frac{aq^{-1-r}-a^{-1}q^{r+1}}{q-q^{-1}}-a^{-1}q^{r+1}.
\]
By induction from the $r=0$ case, this becomes after specialization:
\[
q^{r+2}\nu_x\qn{x-r-1}-\nu_xq^{r}=q^{r}\nu_x(q^2\qn{x-r-1}+1)=q^{r}\nu_x\qn{x-r}.
\]
\end{proof}

\begin{proposition} \label{prop:xWeb_presentation}
The category $\xWeb$ is the braided $Z(q)[\nu_x]$-linear category with objects, upward and downward arrows on a horizontal line with labels in $\N^\ast$. Its morphisms are generated by webs. The braiding is given by \cref{eq:braiding}. Webs are considered up to isotopy and up to the following relations:
\begin{itemize}
    \item upward relations from Definition~\ref{def:iWeb+}. These relations are independent on $x$;
    \item circle evaluation:
\begin{equation}
    \label{eq:unknot}
\begin{tikzpicture}[anchorbase]
\draw [semithick, ->] (0,0) arc (0:360:.5) ;
\node at (.2,-.3) {$k$};
\end{tikzpicture}
\; =\;
\begin{tikzpicture}[anchorbase]
\draw [semithick, ->] (0,0) arc (0:-360:.5) ;
\node at (.2,-.3) {$k$};
\end{tikzpicture}
\; = \;
q^{k(k-1)}\nu_x^k\qbin{k}{x}.
\end{equation}
\item sideway digon removal:
\begin{equation}
    \label{eq:sideway_digon}
    \begin{tikzpicture}[anchorbase]
    \draw [semithick,->] (0,0) -- (0,3) ;
    \draw [semithick,->] (0,2) to [out=90,in=90] (1,2) -- (1,1.5);
    \draw [semithick] (1,1.5) -- (1,1) to [out=-90,in=-90] (0,1);
    \node at (0,-.3) {$k$};
    \node at (0,3.3) {$k$};
    \node at (1.3,1.5) {$l$};
    \node [rotate=90] at (-.3,1.5) {$k+l$}; 
    \end{tikzpicture}
    \; = \;
    \begin{tikzpicture}[anchorbase]
    \draw [semithick,->] (0,0) -- (0,3) ;
    \draw [semithick,->] (0,2) to [out=90,in=90] (-1,2) -- (-1,1.5);
    \draw [semithick] (-1,1.5) -- (-1,1) to [out=-90,in=-90] (0,1);
    \node at (0,-.3) {$k$};
    \node at (0,3.3) {$k$};
    \node at (-1.3,1.5) {$l$};
    \node [rotate=90] at (.3,1.5) {$k+l$}; 
    \end{tikzpicture}
    \; = \;
    q^{-kl}\qbin{k}{k+l}\frac{u_{k+l}}{u_l}.
    \begin{tikzpicture}[anchorbase]
    \draw [semithick,->] (0,0) -- (0,3) ;
    \node at (0,-.3) {$k$};
    \node at (0,3.3) {$k$};
    \end{tikzpicture}
\end{equation}
Above $u_r$ stands for the value of the $r$-labeled unknot.
\item square removal (together with its mirror image):
\begin{equation}
    \label{eq:mixed_square}
    \begin{tikzpicture}[anchorbase,decoration={
    markings,
    mark=at position 0.5 with {\arrow{>}}}]
    \draw[semithick, postaction={decorate}] (0,0) to [out=45,in=180] (.5,.5);
    \draw [semithick ,postaction={decorate}] (.5,.5) -- (1.5,.5);
    \draw [semithick,postaction={decorate}] (.5,1.5) to [out=180,in=180] (.5,.5);
    \draw [semithick, postaction={decorate}] (1.5,1.5) -- (.5,1.5);
    \draw [semithick,postaction={decorate}] (1.5,.5) to [out=0,in=0] (1.5,1.5);
    \draw [semithick,->] (1.5,.5) to [out=0,in=135] (2,0);
    \draw [semithick, postaction={decorate}] (2,2) to [out=-135,in=0] (1.5,1.5);
    \draw [semithick, ->] (.5,1.5) to [out=185,in=-45] (0,2);
    \node at (-.2,-.2) {$1$};
    \node at (-.2,2.2) {$1$};
    \node at (2.2,-.2) {$1$};
    \node at (2.2,2.2) {$1$};
    \node at (1,.2) {$2$};
    \node at (1,1.8) {$2$};
    \end{tikzpicture}
    \; = \;
    \begin{tikzpicture}[anchorbase]
    \draw [semithick, ->] (0,0) to [out=45,in=-45] (0,2);
    \draw [semithick,->] (2,2) to [out=-135,in=135] (2,0);
    \end{tikzpicture}
    \; + \;
    q^2\nu_x \qn{x-2}\;
    \begin{tikzpicture}[anchorbase]
    \draw [semithick, ->] (0,0) to [out=45,in=135] (2,0);
    \draw [semithick,->] (2,2) to [out=-135,in=-45] (0,2);
    \end{tikzpicture}
\end{equation}
\end{itemize}
\end{proposition}

\begin{proof}

 The relations in Definition~\ref{def:iWeb} and Proposition~\ref{prop:xWeb_presentation} do have the same form, so one just has to check that the coefficients match.
 
 The check for the $1$-labeled unknot is very similar to the proof of \cref{lem:qrat_sunknot}. Indeed, we want to see that:
 \[
 \frac{a-a^{-1}}{q-q^{-1}}\xrightarrow{a=q\nu_x\delta_x} \nu_x \qn{x}.
 \]
 Let us rewrite:
 \[
 \frac{a-a^{-1}}{q-q^{-1}}=a^{-1}q\frac{1-a^2}{1-q^2}\xrightarrow{a=q\nu_x\delta_x} \nu_x^{-1}\delta_x^{-1}\qn{x}=\nu_x\qn{x}.
 \]
 Other labelings will follow inductively from this one and the other relations. Indeed, in $\iWeb$, we have:
 \begin{align*}
 \mu_k&=
 \begin{tikzpicture}[anchorbase]
 \draw [semithick,->] (.5,0) arc (0:360:.5);
 \node at (.6,-.2) {$k$};
 \end{tikzpicture}
 =
 \frac{1}{q^{1-k}\qn{k}}
 \begin{tikzpicture}[anchorbase,anchorbase,decoration={markings,mark=at position 0.5 with {\arrow{>}}}]
 \draw [semithick,postaction={decorate}] (.4,.3) arc (38:322:.5);
 \draw [semithick,postaction={decorate}] (.4,-.3) to [out=120,in=-120] (.4,.3);
 \draw [semithick,postaction={decorate}] (.4,-.3) to [out=60,in=-60] (.4,.3);
 \node at (-.6,-.2) {$k$};
 \node [rotate=90] at (.75,0) {\tiny $k-1$};
 \node at (.15,0) {\tiny $1$};
 \end{tikzpicture}
=
 \frac{1}{q^{1-k}\qn{k}}
 \begin{tikzpicture}[anchorbase,anchorbase,decoration={markings,mark=at position 0.5 with {\arrow{>}}}]
 \draw [semithick,postaction={decorate}] (-.4,-.3) arc (-142:142:.5);
 \draw [semithick,postaction={decorate}] (-.4,.3) to [out=-120,in=120] (-.4,-.3);
 \draw [semithick,postaction={decorate}] (-.4,-.3) to [out=-60,in=-90] (0,0) to [out=90,in=60] (-.4,.3);
 \node at (-.7,0) {$k$};
 \node [rotate=90] at (.75,0) {\tiny $k-1$};
 \node at (-.15,0) {\tiny $1$};
 \end{tikzpicture}
 \\
 &=
 \frac{1}{q^{1-k}\qn{k}}\frac{aq^{-k+1}-a^{-1}q^{k-1}}{q-q^{-1}}
 \begin{tikzpicture}[anchorbase]
 \draw [semithick,->] (.5,0) arc (0:360:.5);
 \node at (.6,-.2) {$k-1$};
 \end{tikzpicture}
=\frac{1}{q^{1-k}\qn{k}}\frac{aq^{1-k}-a^{-1}q^{k-1}}{q-q^{-1}} \mu_{k-1}.
 \end{align*}
 Now, we can specialize $a=q\nu_x\delta_x$ and use \cref{lem:formulaqbin} to deduce the following recursive formulas for the specializations $u_{k}$:
 \[
 u_{k}=\frac{1}{q^{1-k}\qn{k}}q^{k-1}\nu_x\qn{x-k+1} u_{k-1}=q^{2k-2}\nu_x\frac{\qn{x-k+1}}{\qn{k}}u_{k-1}.
 \]
 Then one can check that the values:
 \[
u_{k}=q^{k(k-1)}\nu_x^k\qbin{k}{x}
 \]
 are indeed solutions to this recursive equation.

For the sideway digon removal with $k=1$, we recall from \cref{lem:formulaqbin} that: 
\[
\frac{aq^{-1}-a^{-1}q}{q-q^{-1}}
\xrightarrow{a=q\nu_x\delta_x}
q\nu_x\qn{x-1},
\]
which is the expected value for \cref{eq:sideway_digon}. The generalized relations with higher labels are consequences of the $k=1$ case.

Finally, for the mixed square removal relation, we can use once again \cref{lem:formulaqbin} to get that:
\[
\frac{aq^{-2}-a^{-1}q^2}{q-q^{-1}}
\xrightarrow{a=q\nu_x\delta_x}
\nu_x(\qn{x-1}-1)=q^2\nu_x\qn{x-2},
\]
as expected in \cref{eq:mixed_square}.
\end{proof}

\begin{remark}
If $x\in \N$, then $\xWeb$ recovers the category $\nWeb$.

This follows from the fact that:
\[
\delta_n=\frac{q^{2n}-1}{q^2-1}-\frac{q^{2n-2}-1}{q^2-1}=q^{2n-2}.
\]
\end{remark}

With these specialization functors at the ready, we can naturally define the following knot invariants, interpolating the Reshetikhin-Turaev polynomials.

\begin{definition} \label{def:xInv}
  Given $K$ a framed knot, we define $[K]_x$ to be the $a=q\nu_x\delta_x$ specialization of the HOMFLY-PT polynomial of $K$. Equivalently, this can be computed as the value of the knot in $\End_{\xWeb}(\emptyset).$
\end{definition}

\begin{example}
Let us consider the trefoil knot, with coloring $k$. Using the braiding formulas, we get
\begin{equation*}
    \begin{tikzpicture}[semithick,anchorbase,decoration={
    markings,
    mark=at position 0.5 with {\arrow{>}}}]
    \draw[->] (0,0) -- (0.5,0.5);
    \draw (0.5,0) -- (0.3,0.2);
    \draw[->] (0.2,0.3) -- (0,0.5);
    
    \draw[->] (0,1) -- (0.5,1.5);
    \draw (0.5,1) -- (0.3,1.2);
    \draw[->] (0.2,1.3) -- (0,1.5);
    
    \draw[->] (0,2) -- (0.5,2.5);
    \draw (0.5,2) -- (0.3,2.2);
    \draw[->] (0.2,2.3) -- (0,2.5);
    
    \draw (0,0.5)to[out=135,in=-145] (0,1);
    \draw (0.5,0.5)to[out=45,in=-45] (0.5,1);
    
    \draw (0,1.5)to[out=135,in=-145] (0,2);
    \draw (0.5,1.5)to[out=45,in=-45] (0.5,2);
    
    \draw (0,2.5)to[out=135,in=-145] (0,0);
    \draw (0.5,2.5)to[out=45,in=-45] (0.5,0);
    \node at(-0.7,1.5) {$k$};
    \node at(1.2,1.5) {$k$};
    \end{tikzpicture}
    \; = (-q)^{3k^2} \sum_{a,b,c=0}^{k} (-q)^{a+b+c-3k} \;
    \begin{tikzpicture}[semithick,anchorbase,decoration={
    markings,
    mark=at position 0.5 with {\arrow{>}}}]
    \draw[->] (0,-0.2) -- (0,2.5);
    \draw[->] (1,-0.2) -- (1,2.5);
    \draw[postaction={decorate}] (0,0)to (1,0.2);
    \draw[postaction={decorate}] (1,0.4) -- (0,0.6);
    \draw[postaction={decorate}] (0,0.8) -- (1,1);
    \draw[postaction={decorate}] (1,1.2) -- (0,1.4);
    \draw[postaction={decorate}] (0,1.6) -- (1,1.8);
    \draw[postaction={decorate}] (1,2) -- (0,2.2);
    
    \draw (0,2.5)to[out=120,in=-120] (0,-0.2);
    \draw (1,2.5)to[out=60,in=-60] (1,-0.2);
    
    \node at(-0.6,1.5) {$k$};
    \node at(1.6,1.5) {$k$};
    \node at(0.7,0.08) {$\underset{a}{~}$};
    \node at(0.3,0.5) {$\underset{a}{~}$};
    \node at(0.7,0.88) {$\underset{b}{~}$};
    \node at(0.3,1.3) {$\underset{b}{~}$};
    \node at(0.7,1.68) {$\underset{c}{~}$};
    \node at(0.3,2.1) {$\underset{c}{~}$};
    \end{tikzpicture}
\end{equation*}

\noindent For example with $k=1$, the invariant turns out to be
\begin{equation*}
    \begin{tikzpicture}[scale=0.7,semithick,anchorbase,decoration={
    markings,
    mark=at position 0.5 with {\arrow{>}}}]
    \draw[->] (0,0) -- (0.5,0.5);
    \draw (0.5,0) -- (0.3,0.2);
    \draw[->] (0.2,0.3) -- (0,0.5);
    
    \draw[->] (0,1) -- (0.5,1.5);
    \draw (0.5,1) -- (0.3,1.2);
    \draw[->] (0.2,1.3) -- (0,1.5);
    
    \draw[->] (0,2) -- (0.5,2.5);
    \draw (0.5,2) -- (0.3,2.2);
    \draw[->] (0.2,2.3) -- (0,2.5);
    
    \draw (0,0.5)to[out=135,in=-145] (0,1);
    \draw (0.5,0.5)to[out=45,in=-45] (0.5,1);
    
    \draw (0,1.5)to[out=135,in=-145] (0,2);
    \draw (0.5,1.5)to[out=45,in=-45] (0.5,2);
    
    \draw (0,2.5)to[out=135,in=-145] (0,0);
    \draw (0.5,2.5)to[out=45,in=-45] (0.5,0);
    \node at(-0.7,1.5) {$k$};
    \node at(1.2,1.5) {$k$};
    \end{tikzpicture}
    = (-q)^3\Big( -q^{-3}u_1^2 + 3q^{-3}\qn{2}u_2 - 3q^{-3}\qn{2}^2u_2 + q^{-3}\qn{2}^3u_2 \Big)
    = \nu_x^2\qn{x}(q^4 + (1-q^2)\qn{x+1}).
\end{equation*}
\noindent In particular, for $x = 2$, we recover the Jones polynomial. 
\end{example}

\begin{remark}
Contrarily to the HOMFLY-PT polynomial and its Reshetikhin-Turaev specializations, the $x$-invariant of the mirror image of a knot is \emph{not} in general obtained under $q\leftrightarrow q^{-1}$. This is caused by the fact that $\qn{x}$ is in general not symmetric under $q\leftrightarrow q^{-1}$ (even up to a power of $q$). For example, the $x=\frac{2}{3}$ invariant of the trefoil and its mirror image are given below.
\begin{gather*}
    \frac{q^4+q^6+2q^8+2q^{10}}{1+q^2+q^4+2q^6+2q^8+q^{10}} \text{ , } 
    \frac{q^{-2}+q^4+ q^6+2q^8+q^{10}}{1+q^2+q^4+2q^6+2q^8+q^{10}}.
\end{gather*}
\end{remark}

\begin{remark}
One can also define the category $\xWeb$ for $x$ an irrational number, but then the underlying ring is $\Z((q))[\nu_x]$. 
\end{remark}

\begin{remark}\label{rem:normalizedinvariant}
Notice that $\nu_x^{-\mathrm{wr}(K)}[K]_x$ is a genuine knot invariant, with $\mathrm{wr}$ the writhe of $K$ (that is, the number of positive crossings minus the number of negative crossings for any diagram).
\end{remark}

\subsection{Without the quadratic extension}

The following statement makes it easier to reduce to just a rational fraction.
\begin{lemma}\label{lem:homogeneous}
Given a link $L$, the associated $x$-invariant is homogeneous in $\nu_x$.
\end{lemma}
\begin{proof}
This can be deduced from the fact that the HOMFLY-PT polynomial has either only odd or only even powers of $a$. This classical fact can for example be found in~\cite[Remark 2.2]{ChmutovPolyak}.
\end{proof}

One could only consider framed knots or links that produce a homogeneous element of degree $0$ in $\nu_x$. Interestingly, this class does not contain the unknot, but rather versions of it that are stabilized an odd number of time.

\section{Webs, quantized Deligne's category and interpolation between categories of representations}
\label{sec:Deligne}

In this section, we will go back to representation theory. Recall that the computation of the Reshetikhin-Turaev link invariant associated to the quantum group $U_q(\mathfrak{gl}_n)$ ($n=2$ being the Jones polynomial) can be obtained through the category $\nWeb$ as above, but originally relied on the use of $U_q(\mathfrak{gl}_n)$ representations.  To compute these polynomials, one assigns to a tangle a morphism between representations corresponding to the end points of the tangle. The rule to read such representation is to take a tensor product of a copy of the vector representation $\C[q^{\pm 1}]^{n}$ for each upward arrow, and its dual for each downward arrow. As already noted, thanks to works going from Kuperberg~\cite{Kup} to Cautis-Kamnitzer-Morrison~\cite{CKM}, these morphisms are described by webs and one can equivalently perform the computations in the diagrammatic world.

Now, there are two cases where one looses the representation-theoretic background: either for the HOMFLY-PT polynomial, or when we try to replace the parameter $n$ by a real parameter. We will comment on both situations.

For the first one, the category $\iWeb$ provides a handy setting. It is closely related to diagrammatic category $\mathcal{OS}(z,t)$ studied by Brundan in~\cite{Brundan}, where a review of the history of this category is described in Section 1.1.

The second one is actually quite intricate. In the non-quantum setting, Deligne's category $\mathrm{Rep}(GL_t)$ (see~\cite[Chapter 9.12]{Etingof_book} and references therein, notably~\cite{DeligneMilne}) precisely provides an interpolation between the integral values. It contains as a subcategory the so-called walled Brauer algebras. Quantizations of these objects appear in the literature \cite{Kosuda_Murakami,Leduc}, notably in Brundan's work~\cite[Section 1.4]{Brundan} where it is argued that one can take the $a=q^x$ specialization (this is the choice we are willing to circumvent). The notion of quantized oriented Brauer algebra/category is more classical, and properties of these have been extensively studied, for example in~\cite{AST}. However, simply setting $a=q^{x}$ is somewhat undesirable as $q^x$ has no reason to be in $\Q(q)$, and usually is not a well-defined holomorphic function if one desires to specialize the variable $q$ to complex numbers. We thus reformulate the question of defining a quantized Deligne category as follows.

\begin{question}
Does there exist a family of specializations of $\iWeb$ that become $\Q(q)$-linear and so that the unknot's values go to $x$ when $q$ goes to $1$?
\end{question}

This section is devoted to showing that the categories $\xWeb$ provide a satisfying answer, once the question has been slightly loosened to allow to incorporate the parameter $\nu_x$.

\subsection{Structure of the specialization} \label{sec:structure}

To support the above claim, we challenge the choices made throughout this paper, to conclude that they are essentially unique.

Recall that we have defined a $\Q$-indexed family of specializations $\iWeb\rightarrow \xWeb$. In this section, we aim at relating the different specialization functors, for different values of $x$. Here is the strategy: we are looking for categories obtained from $\iWeb$ by taking a quotient of the ring of coefficients $\Q(a,q)$. Let us denote $F_x$ the associated projection functor, and see what relations will imply that the family $\{F_x\}_{x\in \Q}$ is the one we have defined. Each of these functors $F_x$ defines a notion of $q$-deformed $x$ as the image of the stabilized unknot from \cref{lem:qrat_sunknot}. Let us denote these images as $\tqn{x}$.

First, in order to recover the known situation, we impose that at $x\in \N$, the functors $F_n$ are the natural surjections onto $\nWeb$. In other words, we want to have:
\[
F_n\left(\frac{1-a^2}{1-q^2}\right)=\qn{n}.
\]

Let us first focus on the relation between $F_x$ and $F_{x+1}$.\footnote{Note that the restriction function in the non-quantum setting $\mathrm{Rep}(\mathfrak{gl}_x)\rightarrow \mathrm{Rep}(\mathfrak{gl_{x-1}})$ does not extend in the quantum setting to a braided functor, even in the case when $x\in \N$. This prevents us from using this functor as a unification tool, although it would be quite interesting to figure out the natural context in which this functor lives.} Consider the sideway digon removal from~\cref{eq:sideway_digon_1}.  From there, one reads a coefficient $\frac{aq^{-1}-a^{-1}q}{q-q^{-1}}$. When $a=q^n$, this becomes $q^{2-n}\qn{n-1}$, meaning that 
$$
F_n\left(\frac{aq^{-1}-a^{-1}q}{q-q^{-1}}\right) = F_{n-1}\left(\frac{a-a^{-1}}{q-q^{-1}}\right).
$$
\noindent This suggests that for any $x$, the value $F_{x-1}(\frac{a-a^{-1}}{q-q^{-1}})$ should be related to $F_x\left(\frac{aq^{-1}-a^{-1}q}{q-q^{-1}}\right)$.
To bring this closer to our purpose, let us consider the following variation of \cref{eq:sideway_digon_1}:
\[
    \begin{tikzpicture}[anchorbase]
    \draw [semithick,->] (0,0) -- (0,3) ;
    \draw [semithick,->] (0,2) to [out=90,in=90] (1,2);
    \draw [semithick] (1,2) -- (1,1.5) to [out=-90,in=180] (1.25,1.25) to [out=0,in=-90] (1.5,1.5) to [out=90,in=0] (1.25,1.75) to [out=180,in=80] (1.1,1.7);
    \draw [semithick] (1,1.3) -- (1,1) to [out=-90,in=-90] (0,1);
    \node at (0,-.3) {$1$};
    \node at (0,3.3) {$1$};
    \node at (1.3,2.2) {$1$};
    \node at (-.3,1.5) {$2$}; 
    \end{tikzpicture}
    \; = \;
\frac{a^2q^{-2}-1}{q-q^{-1}}.
    \begin{tikzpicture}[anchorbase]
    \draw [semithick,->] (0,0) -- (0,3) ;
    \node at (0,-.3) {$1$};
    \node at (0,3.3) {$1$};
    \end{tikzpicture}
\]
Then evaluating to $a=q^n$ suggests the following identity:
\[
F_x\left(\frac{a^2q^{-2}-1}{q-q^{-1}}\right)=qF_{x-1}\left(\frac{a^2-1}{q^2-1}\right),\]
which in turn implies:
\[
 F_{x}\left(\frac{a^2-1}{q^2-1}\right) = q^2F_{x-1}\left(\frac{a^2-1}{q^2-1}\right) + 1.
\]

\begin{theorem}
There is a unique family of functors $\{F_x\}_{x\in \Q}$ such that 
\begin{itemize}
    \item[(i)] $\forall n \in \N$, $F_n : \iWeb \longrightarrow \nWeb$ are the usual ones (see Definition \ref{def:nWeb}),
    \item[(ii)] $\forall x\in \Q$, $\tqn{x+1} = q^2\tqn{x} + 1$,
    \item[(iii)] There is a relation of the form 
    \[
\widetilde{\left\lbrace\frac{1}{x}\right\rbrace} = q^\alpha \frac{1}{\tqn{x}_{q^{\epsilon}}},\;\alpha\in \Z,\; \epsilon\in \{\pm 1\}.
\]
\end{itemize}
\noindent This unique family is the one we have defined in Definition \ref{def:xWeb}.
\end{theorem}

\begin{proof}
Because the action of $\PSL_2(\Z)$ on the rational line is transitive and generated by the relations $x \mapsto x+1$ and $x \mapsto 1/x$, the values of $\tqn{x}$ for all $x\in \Q$ are uniquely determined for each value of the relation (iii). We are going to show that there is actually one such valid relation. 

An application of the recursion formula (ii) gives:
\[
\widetilde{\left\lbrace\frac{1}{2}\right\rbrace}=q^{-2}\widetilde{\left\lbrace\frac{1}{2}\right\rbrace}-q^{-2}.
\]
Applying the expected relation to both sides of the above equation yields:
\[
q^{\alpha}\frac{1}{\tqn{-2}_{q^\epsilon}}=q^{\alpha-2}\frac{1}{\tqn{2}_{q^\epsilon}}-q^{-2}.
\]
Then one can use a repeated application of the recursion formula to check that:
\[
\tqn{-2}=-q^{-4}\tqn{2},
\]
and from there the relation becomes:
\[
q^{-2}\tqn{2}_{q^\epsilon}=q^{\alpha-2}+q^{\alpha+4\epsilon}.
\]
Replacing $\tqn{2}=1+q^2$, we are left with:
\[
q^{-2}+q^{-2+2\epsilon}=q^{\alpha-2}+q^{\alpha+4\epsilon}.
\]
If $\epsilon=1$, this equation does not have a solution. If $\epsilon=-1$, then this forces $\alpha=0$, so the relation (iii) is imposed to be
$$
\tqn{\frac{1}{x}} = \frac{1}{\tqn{x}_{q^{-1}}}.
$$
This means that the values $F_x\left(\frac{a^2-1}{q^2-1}\right)$ are determined by assumptions (i), (ii) and (iii), and must be $q$-rationals of \cite{MGO-2020}. This determines the specialization $a^2 = (q-1)\qn{x} + 1 = q^2\delta_x$, hence $a = q\nu_x\delta_x$.    
\end{proof}

\subsection{The web categories and related categories}

It turns out that the categories $\xWeb$ have interesting properties, that differ in general from the integral case. Here we review some of them.

\begin{lemma} \label{lem:nbounded}
Setting $\begin{tikzpicture}[anchorbase]
\draw [semithick,->] (0,0) -- (0,.5);
\node at (0,.7) {\vphantom{$k>n$}};
\node at (0,-.2) {$k>n$};
\end{tikzpicture}=0
$
in $\xWeb$ causes torsion in $\End_{\xWeb}(\emptyset)$ unless $x\in \N$ and $x\leq n$.
\end{lemma}

\begin{proof}
Consider the circle evaluation \cref{eq:unknot}. Killing a $k$-labeled circle implies that $\qbin{k}{x}=0$. Taking the $q\rightarrow 1$ limit implies that $\frac{x(x-1)\cdots (x-k+1)}{k!}=0$ which only holds if $x\in \N$ and $k>x$.
\end{proof}

\begin{remark}
The question of the existence of semi-simple quotients would be of major interest, in particular in connection to 3-manifold invariants.
\end{remark}

We now introduce an analog of the quantum Schur algebra, which plays a key role in controlling the representation theory of the quantum group $U_q(\mathfrak{gl}_n)$. Here we use a doubled version akin to the one from~\cite{QS2}.
Let us introduce the $x$-analogue of the doubled Schur algebra as follows:
\begin{definition}
Let $m_1,m_2\in \N$ and $N\in\Z$. We define the algebra $\Schur_2^{x}(m_1,m_2,N)$ as follows:
\[
\Schur_2(m_1,m_2,N)=\END_{\xWeb}\left(
\bigoplus_{\substack{(k_1,\dots,k_{m_1+m_2})\in \N^{m_1}\times (-\N)^{m_2} \\ \sum k_i=N}}
\begin{tikzpicture}[anchorbase]
\draw [semithick, ->] (0,0) -- (0,.5);
\node at (0,-.2) {\small $k_1$};
\node at (1,.25) {$\dots$};
\draw [semithick, ->] (2,0) -- (2,.5);
\node at (2,-.2) {\small $k_{m_1}$};
\draw [semithick, <-] (3,0) -- (3,.5);
\node at (3,-.2) {\small $-k_{m_1+1}$};
\node at (4,.25) {$\dots$};
\draw [semithick, <-] (5,0) -- (5,.5);
\node at (5,-.2) {\small $-k_{m_1+m_2}$};
\end{tikzpicture}
\right).
\]
\end{definition}
Contrarily to the situation for $x=n$, one can prove that this algebra admits no finite-dimensional quotient. We make a precise statement in the special case when $m_1=m_2=1$ and $N=0$.
\begin{proposition} \label{prop:finite_quotients}
Let $x\notin\N$. Then the algebra $\Schur_2^{x}(1,1,0)$ has no finite-dimensional torsion-free quotient. \end{proposition}
\begin{proof}
Consider, for any $k\geq 0$ (if $k=0$, set $\mu_0=1$):
\[
f_k=\begin{tikzpicture}[anchorbase]
\draw [semithick,->] (1,1) to [out=-90,in=-90] (0,1);
\draw [semithick,->] (0,0) to [out=90,in=90] (1,0);
\node at (0,-.2) {\small $k$};
\node at (1,1.2) {\small $k$};
\end{tikzpicture}.
\]
Then we have:
\[
f_kf_l=\delta_{k,l}\mu_k f_k.
\]
So setting $e_k=\frac{f_k}{\mu_k}$ (the denominator is non-zero, thanks to the argument in \cref{lem:nbounded}), we get that the $e_k$'s form a family of orthogonal idempotents. If there is a finite-dimensional quotient, in this quotient one will have:
\[
\exists (i_1,\cdots,i_n),\; \exists (\alpha_{i_1},\cdots,\alpha_{i_n}),\quad \sum_j \alpha_{i_j}\bar{e_{i_j}}=0,
\]
where the bar denotes the projection to the quotient. Now, multiplying this identity with one $\bar{e_{i_j}}$ implies that $\alpha_{i_j}\bar{e_{i_j}}=0$.

Now, one can pre-compose this identity with an $i_j$-labeled cup and post-compose it with a $i_j$-labeled cap to obtain in the quotient:
\[
\alpha_{i_j}
\frac{1}{\mu_{i_j}}\; 
\begin{tikzpicture}[anchorbase]
\draw [semithick,->] (1,0) arc (0:-360:.5);
\draw [semithick,->] (1,-1.5) arc (0:-360:.5);
\node at (1.2,0) {\small $i_j$};
\node at (1.2,-1.5) {\small $i_j$};
\end{tikzpicture}\;=\; \alpha_{i_j}\mu_k=0.
\]
It follows that $\alpha_{i_j}\mu_k=0$ and thus all $\alpha_{i_j}$'s are zero. So no non-trivial combination of the $e_i$'s can be zero, and thus any quotient algebra can only be infinite-dimensional.
\end{proof}

This raises the question of the representation theory behing this web category. It is quite likely that Verma modules should play some role: compare to~\cite[Figure 2]{QS2}.

\section{Behavior of the $x$-polynomial}
\label{sec:derivation}

Quantum rationals are known not be continuous functions in the variable $x$. However, \cite{BBL} show that they are right-continuous for $0<q<1$, see also Proposition 4.6 of \cite{Etingof2025}.

It is very tempting to try to extract information by studying the behavior of the polynomial associated to a knot as $x$ varies, for example in the real interval $[0,1]$. Figure~\ref{fig:plots} presents evaluations of the polynomials for the unknot, the trefoil and the knot $5_1$ at $q=2$, for a sample of values for $x$ in $[0,1]$.

\begin{figure}[H]
    \centering
    \includegraphics[width=5cm]{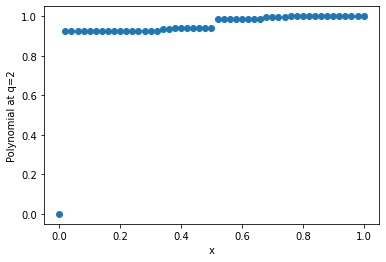} \includegraphics[width=5cm]{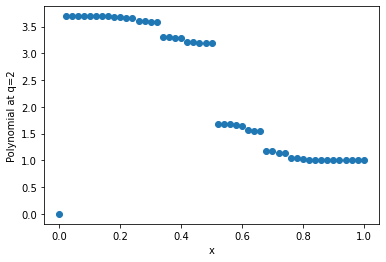} \\
    \includegraphics[width=5cm]{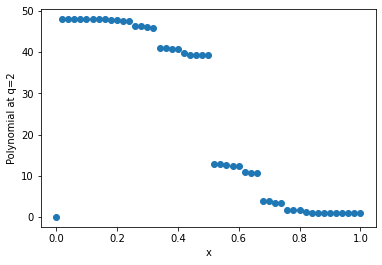} \includegraphics[width=5cm]{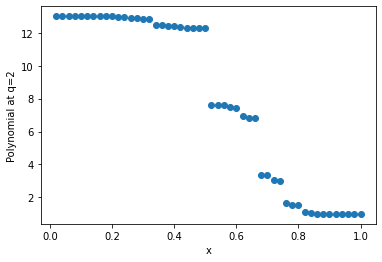}
    \caption{\textbf{Value of the $x$-invariant for $x\in [0,1]$ and $q=2$.} Top left: stabilized unknot. Top right: trefoil. Bottom left: knot $5_1$. Bottom right: ratio between the polynomials for $5_1$ and the trefoil. The fact that there still are discontinuities shows that they are not only controlled by a global factor, which is an incentive to further study these invariants.}
    \label{fig:plots}
\end{figure}

\section{Left $q$-rationals} \label{sec:left}

Another version of $q$-rationals has appeared in work by Bapat, Becker and Licata \cite{BBL}, denoted $\qn{x}^{\flat}$. Those enjoy the same recursive relations, but specialize to different values for $x=n\in \N$. It is quite natural to ask if these deformations could also be used to specialize the HOMFLY-PT polynomial and define web categories.

Let us introduce:
\[
\delta_x^\flat=\qn{x}^\flat-\qn{x-1}^\flat,
\]
and define the quadratic extension $\Z[(q)][\nu_x^\flat]/((\nu_x^\flat)^2=(\delta_x^\flat)^{-1})$.

\begin{definition}
One can define the web categories $\xWeb^\flat$ as the $a=q\nu_x^{\flat}\delta_x^\flat$ specializations of $\iWeb$.
\end{definition}

\begin{remark}
A key point to make is that in this case, one still uses the usual quantum coefficients for upward webs, while mixed relations will use the left versions. Indeed, the upward webs are controlled by the Hecke algebra, which ensures for example that the Reidemeister 2 move holds. Investigating versions of the Hecke algebra based on left quantum integers would also be an interesting problem, which we have not pursued here.
\end{remark}

The $x=n\in \N$ case is of particular interest, and yields invariants $[\cdot]_x^{\flat}$ that are genuinely different from the Reshetikhin-Turaev $U_q(\mathfrak{gl}_n)$ invariants. In particular, a computer search shows the following.

\begin{theorem}
The invariant $[\cdot]_2^\flat$ distinguishes all knots with 10 crossings or less, but the following ones (below, $\#$ stands for mirror image):
\begin{itemize}
    \item $10_{25}$ and $10_{56}$, as well as their mirrors;
    \item $10_{40}^{\#}$ and $10_{103}$, as well as their mirrors;
    \item $8_{8}^{\#}$ and $10_{129}$, as well as their mirrors;
    \item $5_1$ and $10_{132}$, as well as their mirrors;
    \item $8_{16}$ and $10_{156}^{\#}$, as well as their mirrors;
    \item $9_{42}$ and its mirror;
    \item $10_{48}$ and its mirror;
    \item $10_{71}$ and its mirror;
    \item $10_{91}$ and its mirror;
    \item $10_{104}$ and its mirror;
    \item $10_{125}$ and its mirror.
\end{itemize}
\end{theorem}

\begin{proof}
This list has been obtained by computer-assisted computations, using Stoimenow's list of knots and their braid expression as input~\cite{Stoimenow}. Amphichiral knots were filtered out using Stoimenow's list of amphichiral knots~\cite{Stoimenow}.
\end{proof}

\begin{corollary}
On knots with 10 crossings or less, this invariant is as powerful as the HOMFLY-PT polynomial.
\end{corollary}
\begin{proof}
One can check that the pairs above do share the same HOMFLY-PT polynomial.
\end{proof}
The statement above contrasts with the Jones polynomial, that cannot tell apart knots $10_{22}$ and $10_{25}$ for example.

\bibliographystyle{amsalpha}

\newcommand{\etalchar}[1]{$^{#1}$}
\providecommand{\bysame}{\leavevmode\hbox to3em{\hrulefill}\thinspace}
\providecommand{\MR}{\relax\ifhmode\unskip\space\fi MR }
\providecommand{\MRhref}[2]{%
  \href{http://www.ams.org/mathscinet-getitem?mr=#1}{#2}
}
\providecommand{\href}[2]{#2}

\end{document}